\documentclass[a4paper,10pt]{article}
\usepackage{amsmath,amsthm,amsfonts,amssymb,bbm,bm}
\usepackage{enumitem}
\usepackage[utf8]{inputenc}
\usepackage[sort]{natbib} 
\usepackage[usenames, dvipsnames]{color}
\usepackage{colortbl,longtable,multirow}
\usepackage{rotating}
\usepackage{tabularx,booktabs}
\usepackage{pifont}
\usepackage{tikz}
\usepackage{rotating}
\usepackage{colortbl}

\usepackage[pagebackref=true,linktoc=page,colorlinks=true,linkcolor=blue,citecolor=blue]{hyperref}

\theoremstyle{plain}
\newtheorem{thm}{Theorem}[section]
\newtheorem{prop}[thm]{Proposition}

\newtheorem{lem}[thm]{Lemma}
\newtheorem{conj}[thm]{Conjecture}
\newtheorem{prob}{Problem}

\theoremstyle{definition}

\newtheorem{ex}[thm]{Example}

\theoremstyle{remark}
\newtheorem{rem}[thm]{Remark}
\newcommand{\sE}{\mathbb{E}}
\newcommand{\sP}{\mathbb{P}}
\newcommand{\var}{\mathrm{Var}}
\newcommand{\cov}{\mathrm{Cov}}
\newcommand{\rd}{\mathrm{d}}

\newcommand{\RR}{\mathbb{R}}

\newcommand{\Ex}{\text{Ex}}
\newcommand{\Sym}{\text{Sym}}

\title{ 
$\phantom{a}$\\[-2cm] 
Efficient simulation of Brown-Resnick processes\\
based on variance reduction of Gaussian processes
}

\author{Marco Oesting\footnote{University of Siegen, Department Mathematik, 57072 Siegen, Germany, Email: oesting@mathematik.uni-siegen.de} \,\, and \, Kirstin Strokorb\footnote{Cardiff University,  School of Mathematics, Cardiff~CF24\,4AG, UK, Email: StrokorbK@cardiff.ac.uk}}

\date{\today}

\begin{document}

\maketitle

\begin{abstract}
$\phantom{a}$\\[-4mm] 
Brown-Resnick processes are max-stable processes that are associated to Gaussian
processes. Their simulation is often based on the corresponding spectral 
representation which is not unique. We study to what extent simulation accuracy
and efficiency can be improved by minimizing the maximal variance of the 
underlying Gaussian process. Such a minimization is a difficult mathematical 
problem that also depends on the geometry of the simulation domain. We extend 
Matheron's (1974) seminal contribution in two aspects: (i) making his 
description of a minimal maximal variance explicit for convex variograms on 
symmetric domains and (ii) proving that the same strategy reduces the maximal 
variance also for a huge class of non-convex variograms representable through a
Bernstein function. A simulation study confirms that our non-costly modification
can lead to substantial improvements among Gaussian representations. We also compare
 it with three other established algorithms.
\end{abstract}

{\small
  \phantom{\textit{2010 MSC}:} {Secondary 60G60} 
}

\enlargethispage{4mm}
\section{Introduction}

Many powerful tools in geostatistics are conveniently based on Gaussian 
processes as an underlying probabilistic model for uncertainty \citep{cd,
handbookSPATSTAT}. By contrast, assessing the extreme values of spatial data 
genuinely requires statistical methodology that goes beyond such tools. A 
common approach from extreme value analysis is the usage of max-stable models
instead. In particular, the class of Brown-Resnick processes \citep{BR77, 
ksdh09} has emerged as a now widely adopted class of processes considered in 
the analysis of spatial data, cf.\ e.g.\ \cite{river15, buhl16, davhusthib13, 
ekks16, emks15, Bel13, friedrichs17, taperCL14, stein17, thibaud16}. 

There is a strong connection between Brown-Resnick processes and Gaussian
processes: First, Brown-Resnick processes arise as the only possible 
non-degenerate limits of maxima of appropriately rescaled independent Gaussian
processes \citep{ksdh09,kabluchko11}. Second, they can be represented as maxima
of a convolution of the points of a Poisson point process and Gaussian processes.
As the number of Gaussian processes involved in the maximum is locally finite,
Brown-Resnick processes still inherit various properties from Gaussian 
processes on a local level. They are parsimonious models in the sense that 
their law is fully specified by a bivariate quantity, namely the variogram of
the underlying Gaussian process. On the other hand, they are still very 
flexible in the sense that various features such as smoothness, scale or nugget
effect can be controlled by the choice of variogram family. All of this makes
Brown-Resnick processes popular consistent spatial models and marks their 
status as a benchmark in spatial extremes.

To extract probabilistic properties of interest from a fitted Brown-Resnick 
model, it is usually necessary to be able to efficiently simulate from the 
fitted model. Meanwhile, several approaches for this task have been developed.
Starting from the basic threshold stopping approach based on the work of
\cite{schlather02} using plainly the original definition of a Brown-Resnick 
process, \cite{oks12}, \cite{dm15} and \cite{osz18} achieved further 
improvements that are based on modified spectral representations. Based on 
different techniques, \cite{deo16} and \cite{lbdm16} proposed the extremal 
functions and the record-breakers approach, respectively, both of which 
together with the normalization method of \cite{osz18} can now be seen as 
state-of-the-art algorithms for the exact simulation of Brown-Resnick processes. 

When dealing with spatial data, the study area on which the process should be 
simulated may be large with respect to its spatial extent or the number of
locations therein. In such situations, exact simulation of Brown-Resnick 
processes via one of the state-of-art algorithms can be very time-consuming and
depending on the purpose of the application it can be more appropriate to admit
(desirably small) simulation errors. Once a trade-off between accuracy and 
efficiency is necessary, it is no longer clear which of the previously 
considered simulation approaches performs ``best''. 
Keeping this in mind, we return to the initial threshold stopping approach 
devised by \cite{schlather02} from a new perspective in this work. We explore
to what extent a modified choice of Gaussian spectral representation of
Brown-Resnick processes can improve the efficiency or accuracy of the 
simulation and whether an improved threshold
stopping approach can compete with state-of-the-art simulation when an error is
admitted. Dealing with these questions ultimately leads us to a classical
minimization problem for Gaussian processes, namely, to find a Gaussian process 
whose maximal variance across the simulation domain
is minimized, while its variogram on this domain is fixed. In this regard, we
extent Matheron's (\citeyear{matheron74}) seminal contribution in two aspects: 
(i) making his description of a minimal maximal variance explicit for convex 
variograms on symmetric domains and (ii) proving that the same strategy reduces
the maximal variance also for the huge class of non-convex variograms that can
be represented via Bernstein functions. 

The manuscript is organized as follows. Section~\ref{sec:specrep} recalls the 
spectral representation of Brown-Resnick processes and revisits Schlather's 
threshold stopping algorithm for the simulation of max-stable processes on a
compact domain. We explain why it is beneficial for the underlying Gaussian
process to have a reduced maximal variance across the simulation domain. 
Subsequently, Section~\ref{sec:minloggauss}, the main contribution, provides two
results complementing \cite{matheron74} for the corresponding minimization 
problem and elaborates on discretization effects. Throughout the text, the most 
popular family of Brown-Resnick processes that are associated to fractional
Brownian sheets, figure as an example. In Section~\ref{sec:study} we report the
setup and results of a simulation study for this family of models, where we also
compare our approach with three other methods. Finally, we end with a discussion
of our findings in Section~\ref{sec:discussion}. Proofs and additional auxiliary 
results are deferred to Appendix~\ref{sec:proofs}.

\section{Spectral representations and threshold stopping}
\label{sec:specrep}

Let $Z=\{Z(\bm{x})\}_{\bm{x} \in \RR^d}$ be a \emph{simple max-stable process}
on $\RR^d$, which means that for each $n \in \mathbb{N}$ and $n$ i.i.d.\ copies 
$Z_1,\dots,Z_n$ of $Z$, the process $Z_1 \vee \dots \vee Z_n$ of the pointwise 
maxima has the same law as $nZ$ and that $Z$ has \emph{standard Fr\'echet} 
margins: $\sP(Z(\bm{x}) \leq z)=\exp(-1/z)$ for $z>0$, $\bm{x} \in \RR^d$. It 
was shown by \cite{dh84} that, if $Z$ is continuous in probability, there 
exists a non-negative stochastic process $V=\{V(\bm{x})\}_{\bm{x} \in \RR^d}$ 
on $\RR^d$ with $\sE(V(\bm{x}))=1$, $\bm{x} \in \RR^d$, such that the law of 
the simple max-stable process $Z$ is recovered from the following max-series 
\begin{align} \label{eq:spec-rep}
 \{Z(\bm{x})\}_{\bm{x} \in \RR^d} \stackrel{\mathcal{D}}{=} \bigg\{\bigvee_{k=1}^\infty U_k V_k(\bm{x})\bigg\}_{\bm{x} \in \RR^d}.
\end{align}
Here $\{U_k\}_{k=1}^\infty$ denotes a Poisson point process on $(0,\infty)$
with intensity measure $u^{-2} \rd u$, which is independent of the i.i.d.\ 
sequence $\{V_k\}_{k=1}^\infty$ of copies of $V$. The process $V$ is often 
called \emph{spectral process} and the representation \eqref{eq:spec-rep} 
referred to as \emph{spectral representation} of $Z$.
By \cite{ghv90}, if $Z$ has continuous sample paths, the trajectories of $V$
will also be continuous and vice versa, such that the sequence 
$\{(U_k,V_k)\}_{k=1}^\infty$ may be considered as an independently marked
Poisson process $\Pi$ on $(0,\infty) \times C(\RR^d)$ with intensity measure
\begin{align*} 
\Lambda((u,\infty) \times A) = u^{-1} \sP(V \in A), \quad u>0, \ A \subset C(\RR^d) \ \text{Borel},
\end{align*}
where $C(\RR^d)$ is endowed with the usual topology of uniform convergence on
compact subsets. 

One popular choice for the spectral process $V$ of a max-stable process $Z$ is 
a log-Gaussian process of the form 
\begin{align}\label{eq:BRspectralprocess}
V(\bm{x})=e^{\eta(\bm{x})} \quad \text{with} \quad \eta(\bm{x})=W(\bm{x})-\frac{\var(W(\bm{x}))}{2}, \quad \bm{x} \in \RR^d,
\end{align}
where $W$ is a zero-mean Gaussian process $W$ with stationary increments. The 
associated max-stable process $Z$, called \emph{Brown-Resnick process}, was 
introduced and many of its properties were analysed in \cite{ksdh09}.
The requirement that $W$ has \emph{stationary increments} means that the law of
the process $\{W(\bm{x}+\bm{h})-W(\bm{x})\}_{\bm{h} \in \RR^d}$ does not depend
on $\bm{x} \in \RR^d$. Since $W$ is zero-mean Gaussian, it is equivalent to the
\emph{intrinsic stationarity} of the process $W$, that is, the stationarity of 
the process $\{W(\bm{x}+\bm{h})-W(\bm{x})\}_{\bm{x} \in \RR^{d}}$ for all 
$\bm{h} \in \RR^d$, cf.\ e.g.~page~108 in \cite{str13}. It ensures that the 
resulting max-stable process $Z$ is stationary and its law is uniquely specified 
by the variogram
\begin{align*}
\gamma(\bm{x}-\bm{y})=\sE (W(\bm{x})-W(\bm{y}))^2, \quad \bm{x},\bm{y} \in \RR^d.
\end{align*}
\cite{analogies01} show that a function $\gamma:\RR^d \rightarrow \RR$ is a 
(not necessarily centered) variogram of an intrinsically stationary Gaussian
random field $W$ if and only if $\gamma(\bm{0})=0$ and $\gamma$ is 
\emph{negative definite} in the sense that $\gamma(-\bm{x})=\gamma(\bm{x})$ for
$\bm{x} \in \RR^d$ and
\begin{align} \label{eq:NDgamma}
\sum_{i=1}^n \sum_{j=1}^n a_i \gamma(\bm{x}_i - \bm{x}_j) a_j \leq 0
\end{align}
for all finite systems $\bm{x}_1,\dots,\bm{x}_n \in \RR^d$ and $a_1,\dots,a_n \in \RR$ with $a_1+\dots+a_n=0$. 
The most popular family of variograms used in practice is 
$\gamma(\bm{h})= \lVert \bm{h}/s \rVert^\alpha$ for some $\alpha \in (0,2]$ and
$s > 0$ corresponding to \emph{fractional Brownian sheets}. This family of
variograms will also serve as an illustrating example throughout this text. It 
can be shown that all fractional Brownian sheets admit continuous trajectories 
\citep[cf.\ Thm.~1.4.1.\ in][for instance]{adler-taylor}. Hence, the resulting 
max-stable process is also continuous.



Henceforth, we will restrict our attention to a compact subset $K\subset\RR^d$.
Let $\{E_i\}_{i=1}^\infty$ and $\{V_k\}_{k=1}^\infty$ be two independent
i.i.d.\ sequences of standard exponential random variables and copies of $V$, 
respectively, and set $U_k=(\sum_{i=1}^k E_i)^{-1}$, $k=1,2,\dots$.  As the 
points $\{(U_k, V_k)\}_{k=1}^\infty$ form a Poisson point process with 
intensity measure $\Lambda$, the finite approximations
\begin{align*} 
Z^{(n)}(\bm{x}) = \bigvee_{k=1}^n U_k V_k(\bm{x}), \quad \bm{x} \in K,
\end{align*}
almost surely converge to a limit process $Z^{(\infty)}$ (as $n \to \infty$) 
satisfying $Z^{(\infty)} \stackrel{\mathcal{D}}{=} Z$.
To obtain an approximation {built from} a finite number
of exponential random variables $E_k$ and stochastic processes $V_k$, often the process 
$Z^{(T)}$ is considered where $T$ is a stopping time defined as
\begin{align} \label{eq:time}
T = T_{K,\tau} = \min\Big\{k \geq 1: U_{k+1} \tau\leq \inf_{\bm{x} \in K} Z^{(k)}(\bm{x})\Big\}
\end{align}
for some $\tau>0$. 
Such an approximation $Z^{(T)}$ yields an exact simulation of a general max-stable 
process $Z$, i.e.\ $Z^{(T)}=Z^{(\infty)}$ a.s., if the condition 
$\sup_{\bm{x} \in K} V(\bm{x}) < \tau$
is a.s.\ satisfied and for such situations this approach has been proposed in 
\cite{schlather02}. 
In the case when $V$ is log-Gaussian, however, as in 
\eqref{eq:BRspectralprocess}, an approximation error may occur.
The probability that such an error occurs is given by
 \begin{align}
  \notag{} 
\mathcal{P}_{K,\tau} 
 ={}&
\sP(Z^{(T)} \neq Z^{(\infty)} \text{ on } K)
\\
\label{eq:prob-error}
 ={}&  1 - \sE_{Z^{(T)}}\bigg\{\exp\bigg(-\sE_V\bigg(\sup_{\bm{x} \in K} \frac{V(\bm{x})}{Z^{(T)}(\bm{x})} - \sup_{\bm{x} \in K} \frac{\tau}{Z^{(T)}(\bm{x})}\bigg)_+\bigg)\bigg\}
 \end{align}
\citep{os18review}, where the symbol $\sE$ with a stochastic process
as a subscript means that the expectation is meant with respect to this process.

The situation is particulary difficult, when the spectral functions' variance $\var(V(\bm{x}))$ is large, since trajectories $U_kV_k(\bm{x})$ that are generated late in a hypothetical unstopped simulation can have an outsize influence on the sample path, i.e.\ stopping too soon has profound consequences on the quality of the output. Specifically, for Brown-Resnick processes, the variance $\var(V(\bm x))$ of the log-Gaussian spectral functions in \eqref{eq:BRspectralprocess} is of order $\exp({\var(W(\bm{x}))})$ and this order can be excessive depending on the values of $\var(W(\bm{x}))$ on the simulation domain $K$. For instance, if $W$ is the standard (``original'') fractional Brownian sheet that vanishes at the origin $\bm{0}\in \RR^d$ (i.e.\ $W(\bm{0})=0$ a.s.), the variance $\var(V(\bm{x}))$ is exponential in $\lVert \bm{x} \rVert^\alpha$. 
This poses a major challenge for stopping-time based approximate simulation.

What might however mitigate this challenge to some extent is the fact that there
is still some choice among the log-Gaussian representations of Brown-Resnick
processes. While the variogram $\gamma$ uniquely determines the law of the 
associated Brown-Resnick process $Z$, it does not uniquely determine the law of
the Gaussian process $W$ in its spectral representation \eqref{eq:BRspectralprocess}. 
Various covariance functions $C(\bm{x},\bm{y})=\cov(W(\bm{x}),W(\bm{y}))$ on $K$
share the same variogram 
\begin{align*} 
\gamma(\bm{x}-\bm{y})=C(\bm{x},\bm{x}) - 2 C(\bm{x},\bm{y}) + C(\bm{y},\bm{y}),
\quad \bm{x},\bm{y} \in K.
\end{align*}
In what follows, we seek to find such Gaussian representations $W$ for a prescribed variogram $\gamma$, whose variances $\var(V(\bm{x}))$ are uniformly small over the simulation domain $K$. More precisely, we address the following problem in Section~\ref{sec:minloggauss} and study the consequences of using low-variance Gaussian representations in the threshold stopping procedure in Section~\ref{sec:study}.

\begin{prob} \label{prob:minvar}
Let $\mathcal{W}_K(\gamma)$ be the set of Gaussian processes with variogram 
$\gamma$ on the compact simulation window $K$. If it exists, identify a Gaussian
process $W_{\min}\in\mathcal{W}_K(\gamma)$ that minimizes the functional
\begin{align*}
W \mapsto \sup_{\bm{x} \in K} \var(W(\bm{x})).
\end{align*}
\end{prob}

\begin{rem} Addressing Problem~\ref{prob:minvar} will also be beneficial for reducing the error probability $\mathcal{P}_{K,\tau}$ in \eqref{eq:prob-error} for high thresholds $\tau$ or several other error terms in 
\cite{os18review}. To see this, note that Proposition~\ref{prop:variance-tail} in the
Appendix~\ref{sec:proofs} ensures that, for each positive $z \in C(K)$, the tail probabilities
\begin{align*} 
\sP\bigg(\sup_{\bm{x} \in K} \frac{V(\bm{x})}{{z(\bm{x})}} > u\bigg)
= \sP\bigg(\sup_{\bm{x} \in K} \exp\bigg\{ W(\bm{x})-\frac{\var(W(\bm{x}))}{2} - \log z(\bm{x}) \bigg\} > u\bigg)
\end{align*}
decay as fast as possible as $u \uparrow \infty$ if $V$ is built from the solution of Problem~\ref{prob:minvar}.
This property entails that 
\begin{align*} 
\sE \bigg( \sup_{\bm{x} \in K} \frac{V(\bm{x})}{{z(\bm{x})}} - \sup_{\bm{x} \in K} \frac{\tau}{{z(\bm{x})}}\bigg)_+ 
   = \int_{\tau \, \sup_{\bm{x} \in K} {{1/z(\bm{x})}}}^\infty \sP\bigg(\sup_{\bm{x} \in K} \frac{V(\bm{x})}{{z(\bm{x})}} > u\bigg) \, \rd u
\end{align*}
becomes smaller for {each} positive function $z \in C(K)$ and for high thresholds $\tau>0$.
\end{rem}

\section{Minimal log-Gaussian representations} \label{sec:minloggauss}

To the best of our knowledge, Problem~\ref{prob:minvar} has first been 
addressed by \citet{matheron74} who introduced the notion of the minimal
representation of an intrinsically stationary process. Starting from the fact
that -- given a specific centered Gaussian process $W_0$ on $K$ with variogram
$\gamma$ -- any other centered Gaussian process $W$ on $K$
with the same variogram is of the form $W(\bm{x}) = W_0(\bm{x}) + B$ for some 
square-integrable random variable $B$, he defined the minimal representation of 
$W_0$ as the process $W_{\min}(\bm{x}) = W_0(\bm{x}) + B_{\min}$ such that 
\begin{align*}
\sup_{\bm{x} \in K} \var(W_{\min}(\bm{x}))
 = \inf_{B \in L^2(\Omega, \mathcal{A}, \sP)} \sup_{\bm{x} \in K} \var(W_0(\bm{x}) + B), 
\end{align*}
i.e.\ the process whose covariance function $C_{\min}$ is the solution of
Problem~\ref{prob:minvar}.
\citet{matheron74} further showed that, for any sample-continuous 
intrinsically stationary Gaussian process, a unique minimal representation 
exists and has the form
\begin{align}\label{eq:minmeasure}
W_{\min}(\bm{x}) = W_0(\bm{x}) - \int_K W_0(\tilde{\bm{x}}) \, \lambda_{\min}(\rd\tilde{\bm{x}}), \quad \bm{x} \in K,
\end{align}
for some probability measure $\lambda_{\min}$ on $K$. The resulting covariance
function $C_{\min}$ on $K\times K$ can be obtained from
\begin{align} 
\notag 2C_{\min}(\bm{x},\bm{y}) = - \gamma(\bm{x}-\bm{y}) &+ \int_K \gamma(\bm{x}- \tilde{\bm{x}}) \, \lambda_{\min}(\rd\tilde{\bm{x}}) \\
\label{eq:covmin} &+  \int_K \gamma(\bm{y}- \tilde{\bm{y}}) \, \lambda_{\min}(\rd\tilde{\bm{y}})
              - \int_K \int_K \gamma(\tilde{\bm{x}}- \tilde{\bm{y}}) \, \lambda_{\min}(\rd\tilde{\bm{x}}) \, \lambda_{\min}(\rd\tilde{\bm{y}}). 
\end{align}       
Note that the covariance and, consequently, the law of the minimal representation 
does not depend on the initially chosen random field $W_0$.

\subsection{Minimal solution for convex variograms}

If the variogram $\gamma$ is convex and regular in the sense that for two probability measures $\lambda$ and $\mu$ on the compact domain $K$ the equality
\begin{align*}
\int_K \int_K \gamma(\bm{x}-\bm{y}) (\lambda - \mu) (\rd\bm{x}) (\lambda - \mu) (\rd\bm{y}) = 0
\end{align*}
implies $\lambda = \mu$, then \citet{matheron74} characterizes the 
minimizing probability measure $\lambda_{\min}$ by the following two properties
\begin{align}
\label{eq:matheroncond1} \text{Supp}(\lambda_{\min}) &\subset \Ex(K),\\
\label{eq:matheroncond2}
\int_K \gamma(\bm{x}- \tilde{\bm{y}}) \, \lambda_{\min}(\rd\tilde{\bm{y}}) &\leq
\int_K \int_K \gamma(\tilde{\bm{x}}- \tilde{\bm{y}}) \, \lambda_{\min}(\rd\tilde{\bm{x}}) \, \lambda_{\min}(\rd\tilde{\bm{y}}) 
\quad \text{ for all } \bm{x} \in \Ex(K).
\end{align}
Here, $\Ex(K)$ denotes the set of \emph{extremal points} of $K$, i.e.\ the elements of $K$ that cannot be decomposed non-trivially as a convex combination of any two other points of $K$. For instance, the $2^d$ vertices $(\pm R_1,\pm R_2,\dots, \pm R_d)$ of a  $d$-dimensional hyperrectangle $K=\prod_{i=1}^d[-R_i,R_i]$ form its extremal points. In fact, hyperrectangles are the most natural simulation domains that  we consider in practice and we will be mainly interested in this case. For a hyperrectangle $K=\prod_{i=1}^d[-R_i,R_i]$ it is also often convenient to label its vertex set $\Ex(K)$ by subsets $A$ of $\{1,\dots,d\}$ through
\begin{align}\label{eq:vertexlabel}
\Ex\Big(\prod_{i=1}^d[-R_i,R_i]\Big)=\{\bm{v}_A \,:\, A \subset  \{1,\dots,d\}\}, 
   \quad \bm{v}_A=(\sigma^A_i R_i)_{i=1}^d, \quad \sigma^A_i=\begin{cases}+1 & i \in A \\ -1 & i \not\in A\end{cases}.
\end{align}
When the variogram $\gamma$ and the simulation domain $K$ are sufficiently symmetric, the measure $\lambda_{\min}$ can be made even more explicit. To this end, we refer to the set of orthogonal transformations $\bm{M}\in O(\RR^d)$, such that $\bm{M}A=\{\bm{M}\bm{a}\,:\,\bm{a}\in A\}$ coincides with $A$, as the \emph{symmetry group} $\Sym(A)$ of a set $A \subset \RR^d$.

\begin{prop}\label{prop:convex}
Let $W_0$ be a sample-continuous intrinsically stationary Gaussian process on
$\RR^d$ with convex variogram $\gamma(\bm{h})=\psi(\lVert \bm{h}\rVert^2)$, 
$\bm{h} \in \RR^d$. Let $K \subset \RR^d$ be a compact domain, whose 
symmetry group ${\normalfont \Sym}({\normalfont \Ex}(K))$ acts transitively on the set of extremal points ${\normalfont \Ex}(K)$. Then the minimizing measure $\lambda_{\min}$ in the sense of \eqref{eq:minmeasure} is the uniform distribution on ${\normalfont \Ex}(K)$.
\end{prop}

\begin{ex} \label{example:fBMconvex}
Let $W_0$ be a fractional Brownian sheet on $\RR^d$ with variogram 
$\gamma(\bm{h})=\lVert \bm{h}/s\rVert^\alpha$ for some $\alpha \in [1,2)$, 
$s>0$. Let $K=\prod_{i=1}^d[-R_i,R_i]$ be a $d$-dimensional hyperrectangle,
whose vertices $\bm{v}_A$ are labelled by subsets $A \subset \{1,\dots,d\}$ as
in \eqref{eq:vertexlabel}. Then Proposition~\ref{prop:convex} applies and says
that the process
\begin{align} \label{eq:fBMconvex}
W(\bm{x}) = W_0(\bm{x}) - \frac{1}{2^d} \sum_{A \subset \{1,\dots,d\}} W_0(\bm{v}_A), \quad \bm{x} \in K,
\end{align}
possesses the smallest maximal variance on $K$ among the family of intrinsically stationary 
Gaussian processes with variogram $\gamma$.
\end{ex}

\begin{rem}
\cite{matheron74} provides the argument for  \eqref{eq:matheroncond1} and \eqref{eq:matheroncond2} being necessary and sufficient for $\lambda_{\min}$ to be the minimal measure for compact and convex sets $K$. It is however easily checked that convexity of $K$ can be dispensed with. 
\end{rem}

\subsection{Reduction of maximal variance for non-convex variograms}

For non-convex variograms it is not so clear how to obtain the minimizing
measure $\lambda_{\min}$ explicitly, not even for hyperrectangles. However, in
many situations it is still possible to apply the same strategy as in 
Example~\ref{example:fBMconvex} to substantially reduce the maximal variance. 
At least we show below that this is possible when the variogram $\gamma$ can be 
represented as $\gamma(\bm{h})=\psi(\lVert \bm{h} \rVert^2)$, $\bm{h}\in\RR^d$,
for a Bernstein function $\psi$. 

There are several definitions of Bernstein functions and various properties and 
examples have been summarized in the recent monograph \cite{schillingBernstein}.
For us it will be convenient to define a \emph{Bernstein function} as a function
$\psi:\RR_+\rightarrow \RR$ on the positive real line $\RR_+=[0,\infty)$ that is
bounded from below and \emph{negative definite} in the sense that
\begin{align}\label{eq:NDpsi}
\sum_{i=1}^n \sum_{j=1}^n a_i \psi(s_i+s_j) a_j \leq 0
\end{align}
for all finite systems $s_1,\dots,s_n \in \RR_+$ and $a_1,\dots,a_n\in \RR$ with 
$a_1+\dots+a_n=0$, cf.~page~113/114 of \cite{bcr84}. Note that we deviate from 
most of the literature where continuity at $0$ is additionally required. The 
following lemma shows that considering only variograms that can be represented 
by Bernstein functions is not a serious restriction and, secondly, that Bernstein
functions obey certain monotonicity properties defined as follows: A function 
$f:\RR_+\rightarrow\RR$ is \emph{$n$-alternating} if it satisfies 
\begin{align}\label{eq:CApsi} 
\Delta_{s_1}\Delta_{s_2}\dots \Delta_{s_n}f(s) = \sum_{A \subset \{1,\dots,n\}} (-1)^{\lvert A \rvert} f\bigg(s + \sum_{i\in A} s_i\bigg) \leq 0
\end{align} 
for $s,s_1,s_2,\dots,s_n \in \RR_+$. In particular, $f$ is $1$-alternating
if and only if it is non-decreasing.

\begin{lem} \label{lemma:Bernstein}
Let $\psi:\RR_+ \rightarrow \RR$ be bounded from below. Then the following 
statements are equivalent.
\begin{enumerate}[label={(\roman*)},itemsep=0mm]
\item $\psi$ is a Bernstein function. 
\item $\psi(\lVert \bm{h} \rVert^2)$, $\bm{h} \in \RR^d$, is negative definite 
in the sense of \eqref{eq:NDgamma} for any dimension $d \in \mathbb{N}$.
\item $\psi$ is $n$-alternating of any order $n \in \mathbb{N}$.
\end{enumerate}
\end{lem}

In particular, the equivalence of (i) and (ii) means that $\psi$ is a Bernstein function if and only if  $\gamma(\bm{h}) = \psi(\lVert \bm{h} \rVert^2)$, $\bm{h} \in \RR^d$, is a valid (not necessarily centred) variogram in any dimension $d \in \mathbb{N}$.  Finally, this allows us to transfer the strategy from Example~\ref{example:fBMconvex} 
to non-convex variograms as follows.

\begin{prop} \label{prop:non-convex}
Let $W_0$ be an intrinsically stationary Gaussian process on $\RR^d$ with 
variogram $\gamma(\bm{h})=\psi(\lVert \bm{h}\rVert^2)$, $\bm{h} \in \RR^d$, for
a Bernstein function $\psi$, and $W_0(\bm{0})=0$ almost surely. Let 
$K=\prod_{i=1}^d[-R_i,R_i]$ be a $d$-dimensional hyperrectangle, whose vertices
$\bm{v}_A$ are labelled by subsets $A \subset \{1,\dots,d\}$ as in 
\eqref{eq:vertexlabel}. Then the maximal variance of the modified process
 $W$ in \eqref{eq:fBMconvex}
is at least as small as the maximal variance of the original process $W_0$ on 
the domain $K$, i.e.
\begin{align*}
\sup_{\bm{x} \in K} \var(W(\bm{x}))
\leq \sup_{\bm{x} \in K} \var(W_0(\bm{x})).
\end{align*}
\end{prop}

\begin{ex} \label{example:fBMnonconvex}
Proposition~\ref{prop:non-convex} applies to the situation of 
Example~\ref{example:fBMconvex} with $\alpha \in [1,2)$ replaced by 
\emph{arbitrary} $\alpha \in (0,2)$ and $W_0(\bm{0})=0$ almost surely. In
particular, the maximal variance on the hyperrectangle $K$ can also be reduced
for $\alpha \in (0,1)$ by the same simple trick of subtracting the vertices of
$K$ with equal weights that we applied already for $\alpha \in [1,2)$ in 
Example~\ref{example:fBMconvex}. Summarizing, we see that this trick 
\emph{reduces} the maximal variance on $K$ for \emph{any} fractional Brownian 
sheet. For $\alpha \in [1,2)$ it even \emph{minimizes} this maximal variance, 
cf.~Example~\ref{example:fBMconvex}.
\end{ex}

\begin{rem}
\begin{enumerate}[label={(\alph*)}, wide=0mm, itemsep=0mm]
\item In fact, the proof of Proposition~\ref{prop:non-convex} only requires
$n$-alternation of $\psi$ for $n \leq 3$ and not necessarily that $\psi$ is a
Bernstein function.
\item In Examples~\ref{example:fBMconvex}, \ref{example:fBMnonconvex} and 
Proposition~\ref{prop:non-convex} the variogram $\gamma(\bm{h})$ can be replaced
by $\tilde \gamma(\bm{h}) = \gamma(\bm{M} \bm{h})$, where $\bm{M}$ is any 
invertible symmetric $d\times d$ matrix. This is possible, since such a 
transformation of $\RR^d$ maps hyperrectangles that are centered at the origin 
again into hyperrectangles that are centered at the origin, whence 
Propositions~\ref{prop:convex} and \ref{prop:non-convex} can be applied to 
the transformed processes $W_0(\bm{M}^{-1} \bm{x})$, $W(\bm{M}^{-1} \bm{x})$, 
$\bm{x} \in \RR^d$, and their variogram 
$\tilde \gamma(\bm{M}^{-1} \bm{h}) = \gamma(\bm{h})$, $\bm{h} \in \RR^d$.
\end{enumerate}
\end{rem}

\subsection{Minimal $K$-stationary representations} 

There are some situations, in which we can even compare the improvements of the
variance reduction of Proposition~\ref{prop:non-convex} to the true minimal 
maximal variance on the simulation domain $K$. This is the case when the 
minimizing measure $\lambda_{\min}$ leads to a \emph{$K$-stationary} solution, 
that is, when $C_{\min}(\bm{x},\bm{y})$ in \eqref{eq:covmin} only depends on 
$\bm{x}-\bm{y}$ for $\bm{x}, \bm{y} \in K$. In general, it is not clear that a
given variogram $\gamma(\bm{x}-\bm{y})$, $\bm{x},\bm{y} \in K$ possesses a 
$K$-stationary representation at all, cf.\ Remark~3.2.5 in \cite{bcr84} for a 
counterexample. For fractional Brownian sheets on a $d$-dimensional Euclidean ball 
$B_R(\bm{0})=\{\bm{h} \in \RR^d \,:\, \lVert \bm{h} \rVert \leq R \}$ however, the following proposition restating \cite{gneiting2000addendum} and 
\cite{matheron74} confirms their existence and makes their covariance explicit. 

\begin{prop}\label{prop:locstat}
\begin{enumerate}[label={(\alph*)},wide=0mm]
\item{}[\cite{gneiting2000addendum}] For $\alpha \in (0,2)$, $s>0$ and $R>0$ 
the function
\begin{align*}
C(\bm{x},\bm{y}) = a - \frac{1}{2}\big\lVert (\bm{x} - \bm{y})/s \big\rVert^\alpha, \quad \bm{x},\bm{y} \in B_R(\bm{0})
\end{align*}
is a covariance function if and only if $a \geq (A_{\alpha,d}/2)(R/s)^\alpha$, 
where
\begin{align*}
A_{\alpha,d} = {\Gamma\bigg(\frac{2-\alpha}{2}\bigg)\Gamma\bigg(\frac{d+\alpha}{2}\bigg)}{\Gamma\bigg(\frac{d}{2}\bigg)^{-1}}. 
\end{align*}
\item{}[\cite{matheron74}] Choosing $a=(A_{\alpha,d}/2) (R/s)^\alpha$ minimizes
the maximal variance 
\begin{align*} \sup_{\bm{x} \in B_R(\bm{0})} C(\bm{x},\bm{x}) = \sup_{\bm{x} \in B_R(\bm{0})} \var(W(\bm{x}))
\end{align*} 
among all intrinsically stationary Gaussian representations $W$ on the domain $B_R(\bm{0})$ 
for the variogram $\gamma(\bm{h})=\lVert\bm{h}/s\rVert^\alpha$ if and only if
$d=1$ and $\alpha \in (0,1]$.
\end{enumerate}
\end{prop}
Generally, the choice in (b) only minimizes the maximal variance among the 
$K$-\emph{stationary} Gaussian representations of the variogram  
$\gamma(\bm{h})=\lVert\bm{h}/s\rVert^\alpha$ on the domain $K=B_R(\bm{0})$, not 
necessarily among \emph{all} intrinsically stationary Gaussian representations
of $\gamma$ on $B_R(\bm{0})$.

\begin{figure}
\centering
\begin{minipage}{12cm}
\centering
{\small 
\tabcolsep0mm
\begin{tabular}{p{1mm}ccc}
&
\begin{minipage}{4cm}\centering $\bm{\alpha = 0.7}$\end{minipage}&
\begin{minipage}{4cm}\centering $\bm{\alpha = 1.0}$\end{minipage}&
\begin{minipage}{4cm}\centering $\bm{\alpha = 1.3}$\end{minipage}
\end{tabular}
}
\includegraphics[width=12cm]{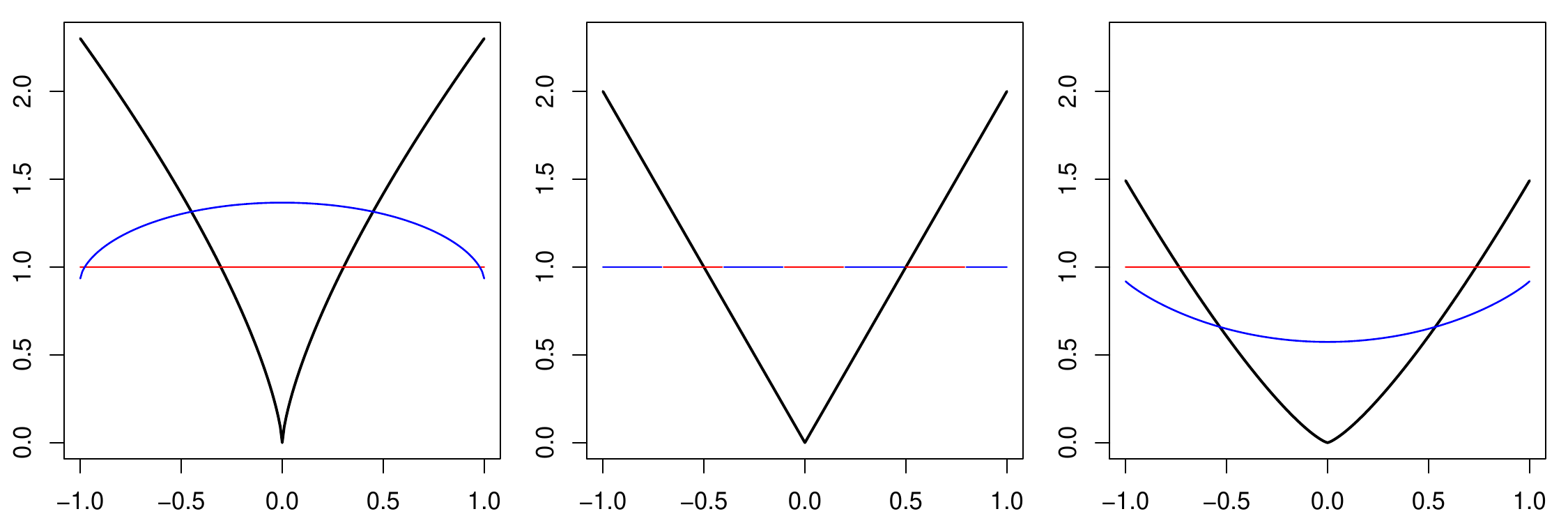}
\end{minipage}
\caption{\small Variances $\sigma^2(t)$ of the Gaussian representations of the 
   variogram $\gamma({h})=\lvert h/s\rvert^\alpha$ on the domain $K=[-1,1]$. 
   The plots show the variance for the original representation with $W_0(0)=0$
   (black), the minimal $K$-stationary representation (red) and the 
   $\lambda$-modified representation with $\lambda=\text{Unif}(\Ex(K))$ (blue).
   For $\alpha=1$ the last two coincide. The scale $s>0$ is chosen such that 
   the variance of the minimal $K$-stationary representation (red) is
   normalized to 1.}
\label{fig:Dim1variances}
\vspace{5mm}
\mbox{
\begin{minipage}{1cm}
\rotatebox{90}{\small 
\tabcolsep0mm
\begin{tabular}{ccc}
\begin{minipage}{4.5cm}\centering $\bm{\lambda}$\textbf{-modified} \end{minipage}&
\begin{minipage}{4.5cm}\centering $K$\textbf{-stationary}\end{minipage}&
\begin{minipage}{4.5cm}\centering \textbf{original} \end{minipage}
\end{tabular}
}
\end{minipage}\hspace*{-9mm}
\begin{minipage}{14cm}
\centering
{\small 
\tabcolsep0mm
\begin{tabular}{ccc}
\begin{minipage}{4.5cm}\centering $\bm{\alpha = 0.7}$\end{minipage}&
\begin{minipage}{4.5cm}\centering $\bm{\alpha = 1.0}$\end{minipage}&
\begin{minipage}{4.5cm}\centering $\bm{\alpha = 1.3}$\end{minipage}
\end{tabular}
}
\includegraphics[width=4.5cm]{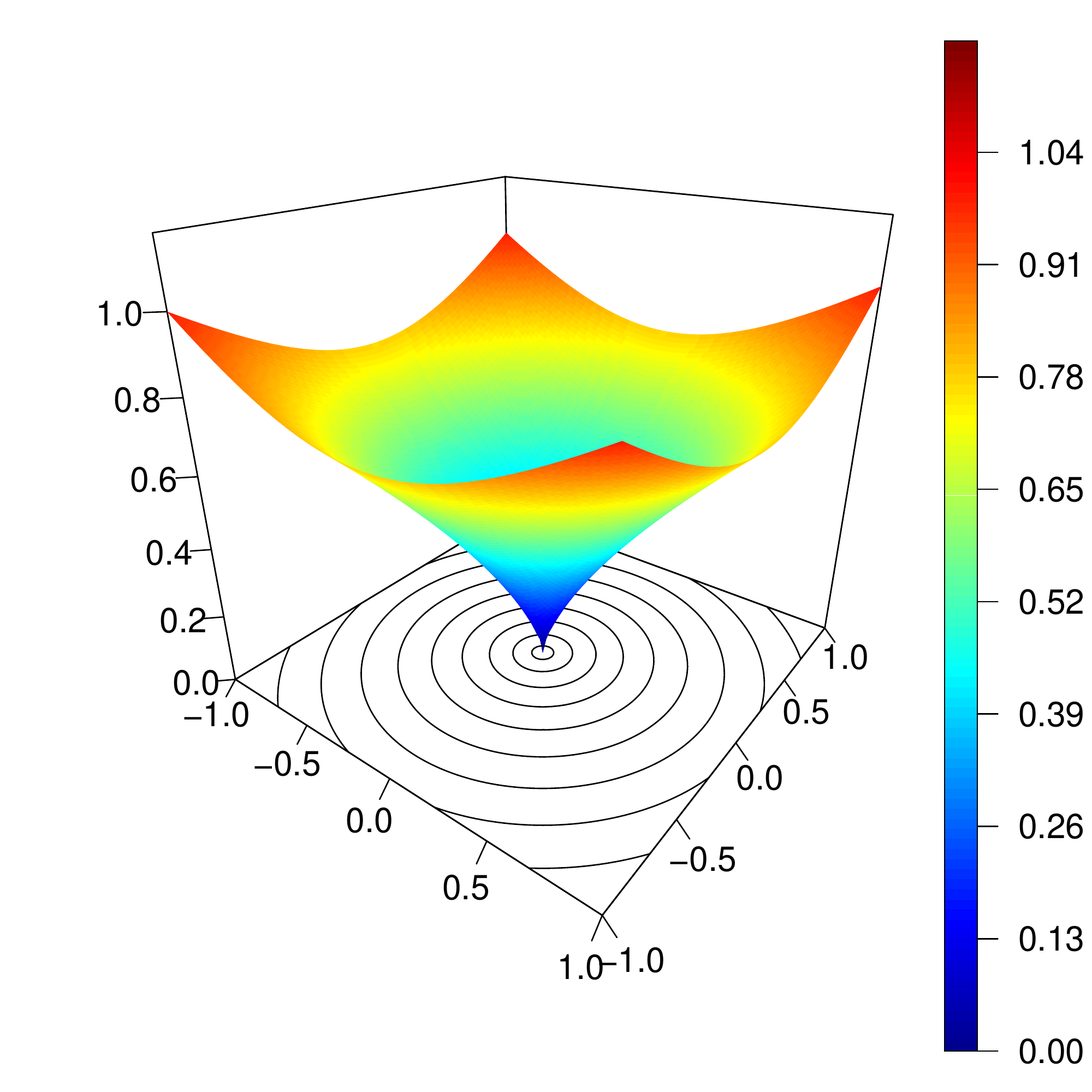}
\includegraphics[width=4.5cm]{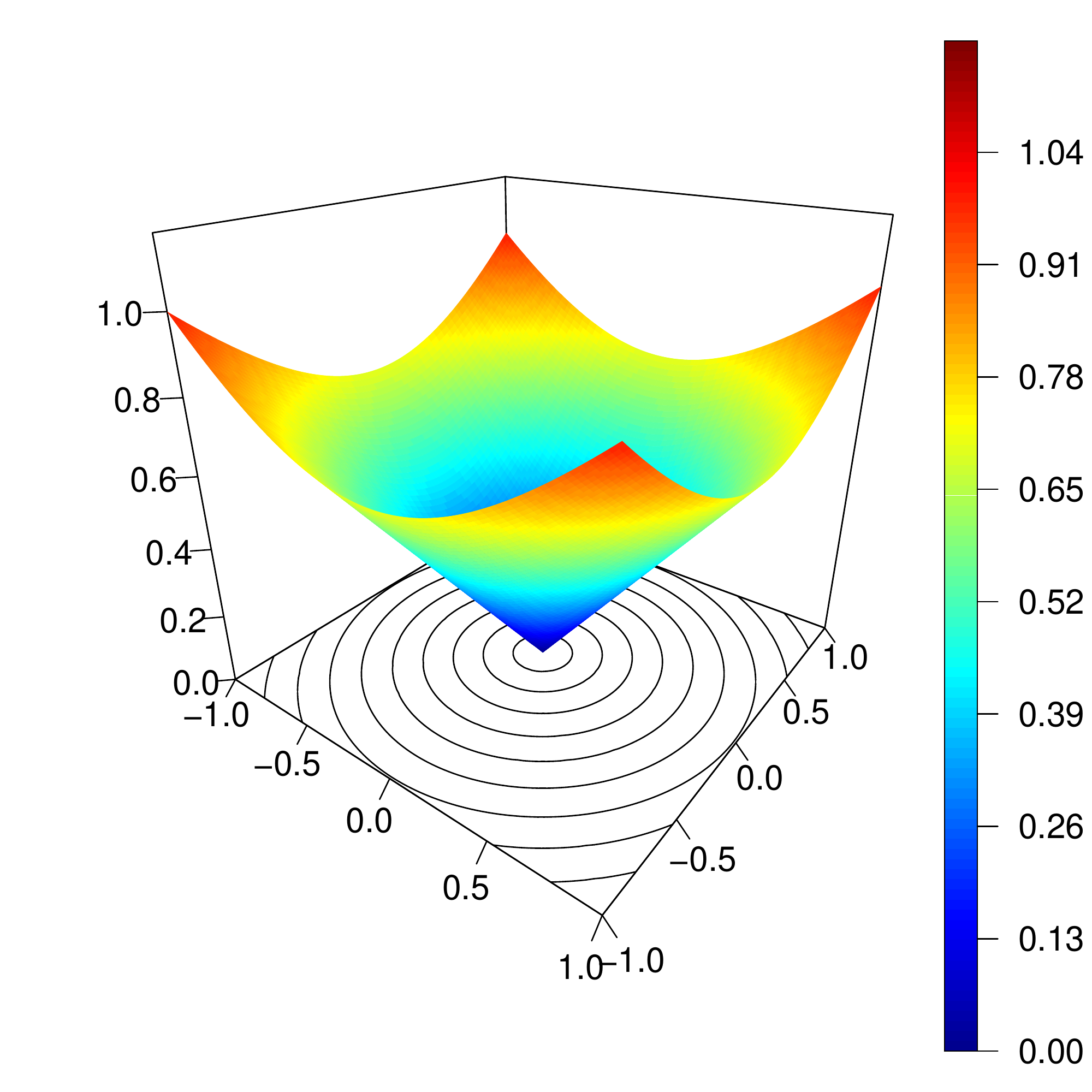}
\includegraphics[width=4.5cm]{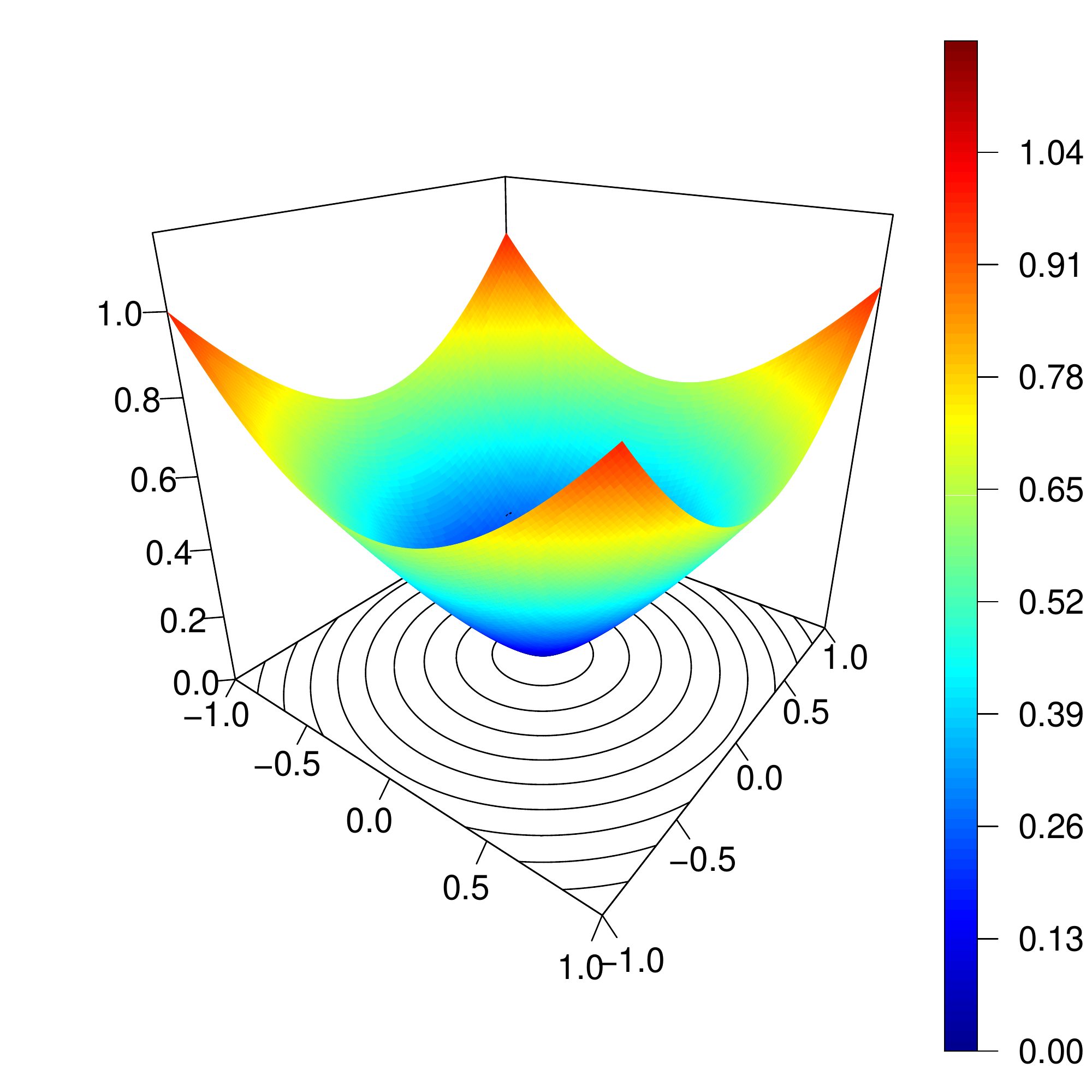}\\
\includegraphics[width=4.5cm]{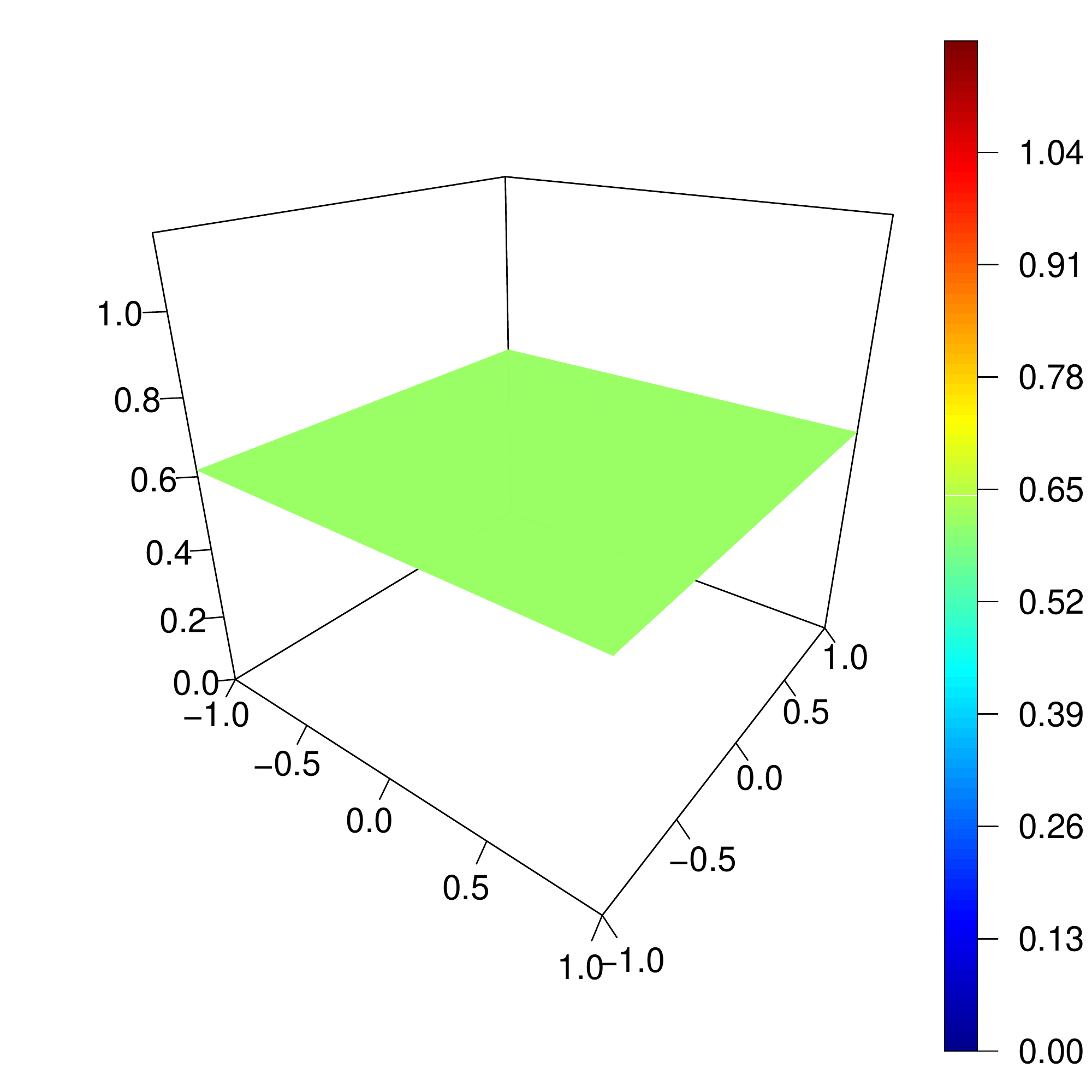}
\includegraphics[width=4.5cm]{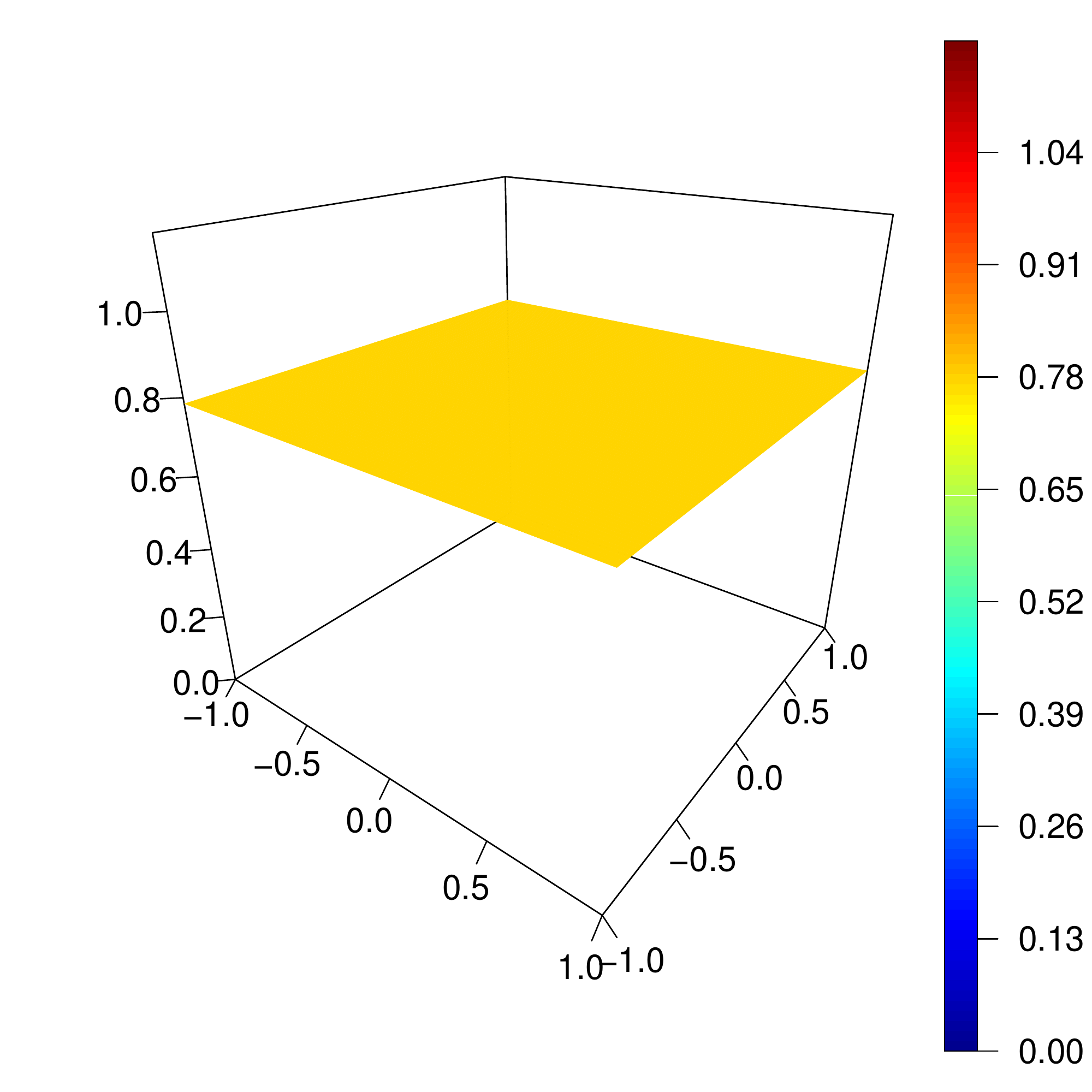}
\includegraphics[width=4.5cm]{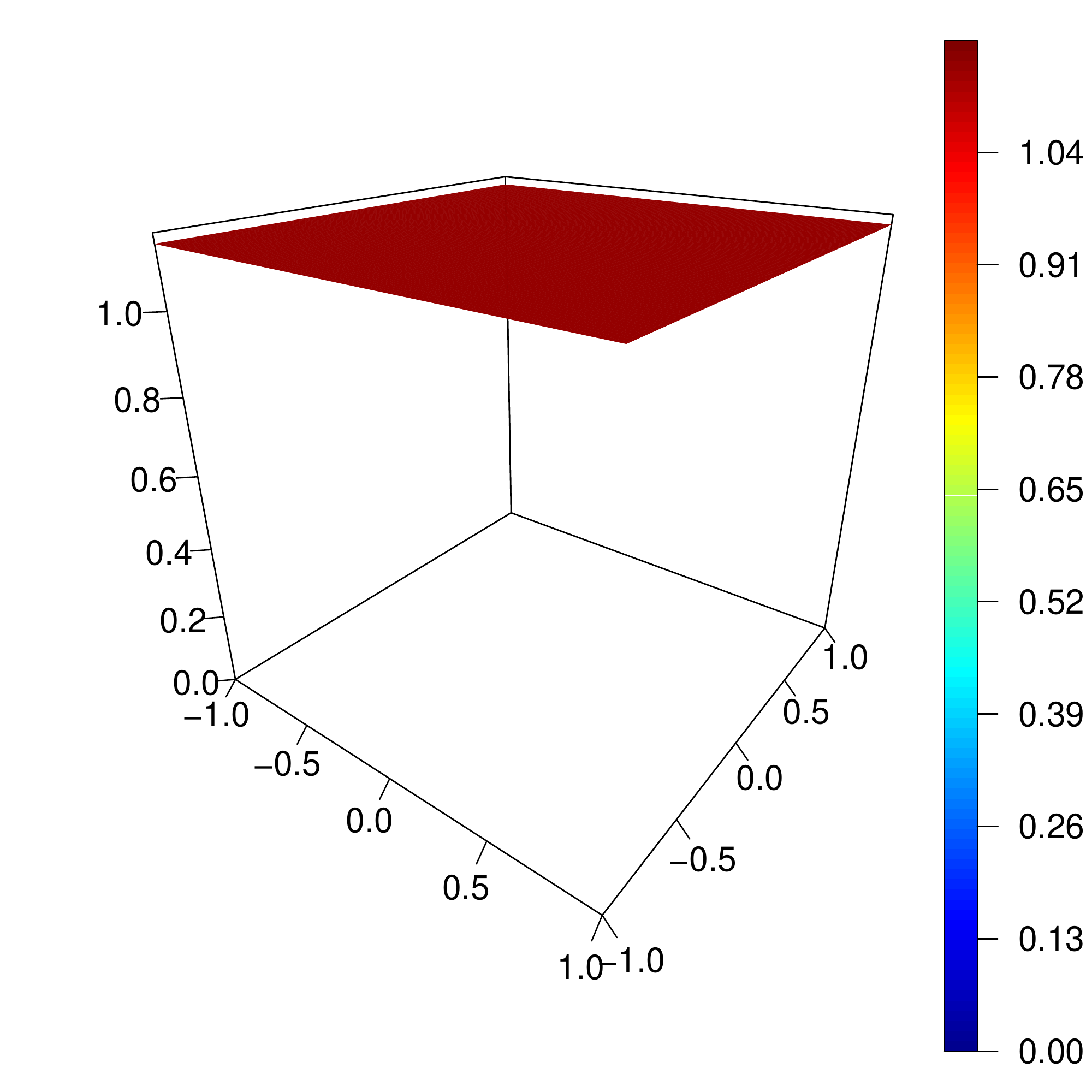}\\
\includegraphics[width=4.5cm]{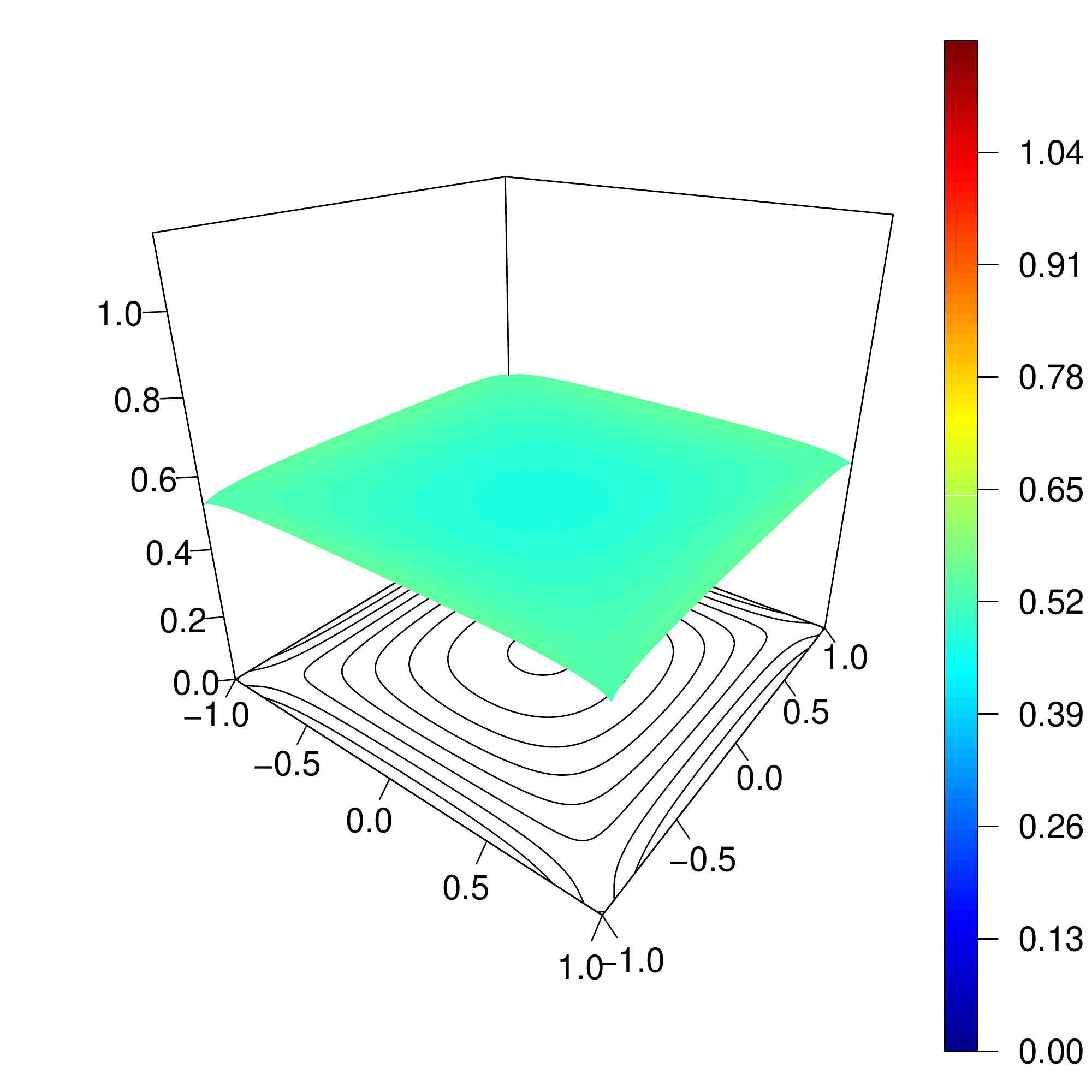}
\includegraphics[width=4.5cm]{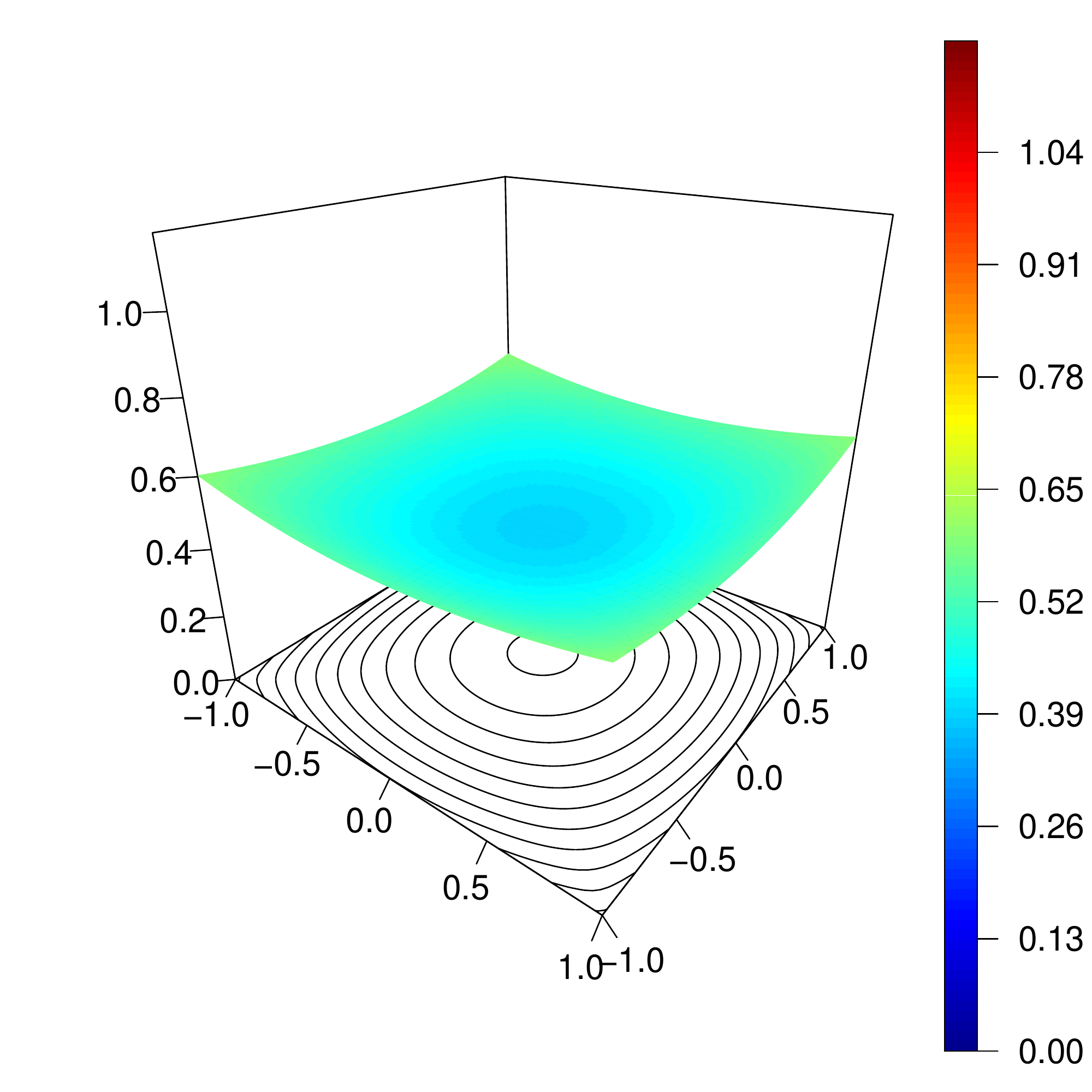}
\includegraphics[width=4.5cm]{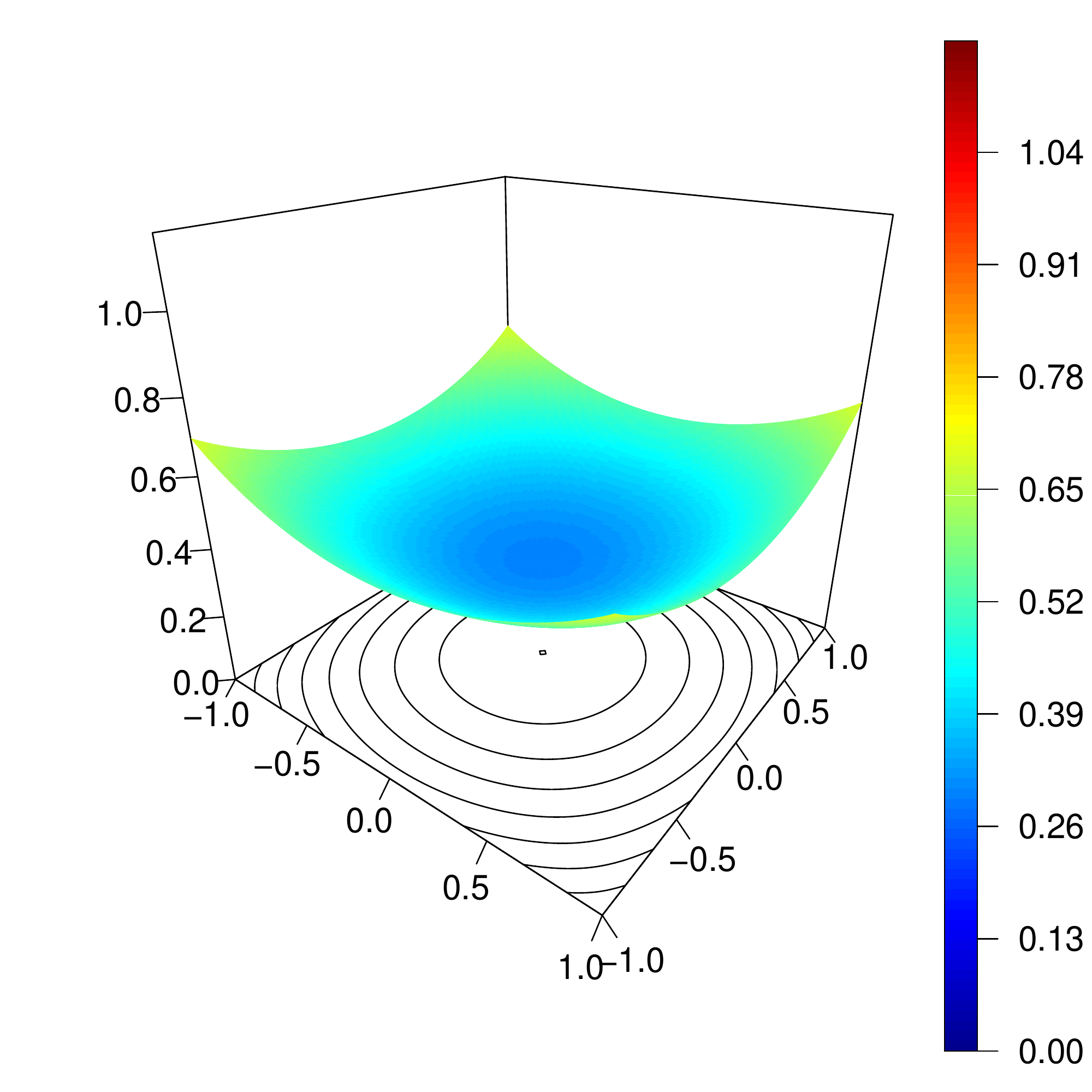}
\end{minipage}
}
\caption{\small Variances $\sigma^2(t)$ of the Gaussian representations of the 
         variogram $\gamma(\bm{h})=\lVert \bm{h}/\sqrt{2}\rVert^{\alpha}$ on 
         the domain $K=[-1,1]^2$ for $\alpha\in\{0.7, 1.0, 1.3\}$ (left to 
         right). The plots show the variance for the original representation
         with $W_0(\bm{0})=0$, the minimal $K$-stationary representation
         and the $\lambda$-modified representation with 
         $\lambda=\text{Unif}(\Ex(K))$ (top to bottom). Minimality of the 
         $K$-stationary representation refers to the minimal ball 
         $B_{\sqrt{2}}(\bm{0})$ containing $K$.}
\label{fig:Dim2variances_a071013_overview_orig_locstat_modified}
\end{figure}

\begin{ex} \label{example:fBMgeneral}
Let $\gamma({h})=\lvert {h}/s\rvert^\alpha$, $\alpha \in (0,2)$, $s>0$, be the 
family of fractional Brownian motion variograms in dimension $1$. Collectively,
the results from Examples~\ref{example:fBMconvex} and \ref{example:fBMnonconvex},
Proposition~\ref{prop:locstat} (b) provide a full description of the 
intrinsically stationary Gaussian representations $W_{\min}$ of $\gamma$ on the
domain $[-R,R] \subset \RR$ that minimize the maximal variance therein. 
\begin{itemize}[itemsep=0mm,leftmargin=\parindent]
\item For $\alpha \in [1,2)$ it is given by the minimizing measure
 $\lambda_{\min}=\frac{1}{2}(\delta_{-R}+\delta_R)$ as described in 
 Proposition~\ref{prop:convex} and Example~\ref{example:fBMconvex}. If $W_0$ is
 any intrinsically stationary Gaussian representation of $\gamma$, it can be 
 obtained as the \emph{$\lambda_{\min}$-modified process}
\begin{align} \label{eq:Wmin1}
W_{\min}(x) = W_0(x) - \frac{1}{2} (W_0(-R)+W_0(R)), \quad x \in [-R,R].
\end{align} 
\item For $\alpha \in (0,1]$ Proposition~\ref{prop:locstat} (b) tells us that
the covariance function of $W_{\min}$ is
\begin{align} \label{eq:Cmin1}
C_{\min}(x,y)=\frac{A_{\alpha,1} - \lvert x-y \lvert^\alpha}{2s^\alpha}, \quad x,y \in [-R,R],
\end{align}
which is the minimal $K$-\emph{stationary} representation of $\gamma$ on
$K=B_R(0)=[-R,R]$.
\end{itemize}
This description was already obtained by \cite{matheron74}. In particular, the
case $\alpha=1$ in dimension $1$ is very special in the sense that (i) 
$d$-dimensional balls and $d$-dimensional hyperrectangles are the same and (ii)
the functions $h \mapsto \lvert h/s \rvert^\alpha$, $s>0$, are both
concave and convex on $[0,\infty)$ for $\alpha=1$. This leads to the situation that the 
minimizing representation can be obtained either way. The covariance of 
\eqref{eq:Wmin1} is precisely \eqref{eq:Cmin1} for $\alpha=1$.

As an illustration of the different cases, Figure~\ref{fig:Dim1variances} 
shows the variances of the original representation (with $W_0(0)=0$), the
$\lambda$-modified representation (with $\lambda=\frac{1}{2}(\delta_{-1}+\delta_1)$)
and the minimal $K$-stationary representation of Gaussian random fields
on $K=[-1,1]$ with variogram $\gamma(h)=\lvert h/s\rvert^\alpha$ for
$\alpha \in \{0.7,1,1.3\}$.
\end{ex}

\begin{ex} \label{example:fBSgeneral}
Let $\gamma(\bm{h})=\lVert \bm{h}/s\rVert^\alpha$, $\alpha \in (0,2)$, $s>0$, be
the family of variograms of fractional Brownian sheets in dimension $d \geq 2$ 
and consider the (hyper-)rectangular simulation domain $K= \prod_{i=1}^d [-R_i,R_i]$. 
Again we are interested in the intrinsically stationary Gaussian representation
$W_{\min}$ of $\gamma$ on the domain $K \subset \RR^d$ that minimize the maximal 
variance therein.
\begin{itemize}[itemsep=0mm,leftmargin=\parindent]
\item When $\alpha \in [1,2)$, still Proposition~\ref{prop:convex} and 
Example~\ref{example:fBMconvex} provide the  minimizing measure 
\begin{align*}
\lambda_{\min} = \frac{1}{2^d} \sum_{A \subset \{1,\dots,d\}} \delta_{\bm{v}_A},
\quad \bm{v}_A=(\sigma^A_i R_i)_{i=1}^d, \quad \sigma^A_i=\begin{cases}+1 & i \in A \\ -1 & i \not\in A\end{cases},
\end{align*}
and $W_{\min}$ can be obtained as the $\lambda$-\emph{modified} process from
\eqref{eq:fBMconvex}.
\item However, for $\alpha \in (0,1)$, we do not know the minimizing measure $\lambda_{\min}$
or the corresponding covariance $C_{\min}$, not even if we replace the domain $K$ by
the $d$-dimensional ball $B_R(\bm{0})$.  At least, Proposition~\ref{prop:locstat}
identifies on $\widetilde K=B_R(\bm{0})$ the minimal $\widetilde K$-\emph{stationary} 
Gaussian representations of the variogram $\gamma$ for $\alpha \in (0,2)$. 
For $R=(R_1^2+\dots+R_d^2)^{1/2}$ the domain $\widetilde K = B_R(\bm{0})$ is the
smallest ball that contains $K$. Therefore, the covariance
\begin{align}\label{eq:locstatK}
C(\bm{x},\bm{y}) = \frac{A_{\alpha,d}R^\alpha - \big\lVert (\bm{x} - \bm{y}) \big\rVert^\alpha}{2s^\alpha} , \quad \bm{x},\bm{y} \in K,
\end{align}
also describes a $K$-stationary Gaussian representation of $\gamma$, but we
do not know if it is minimal among $K$-stationary Gaussian representations of
$\gamma$ on $K$. However, Figure~\ref{fig:Dim2variances_a071013_overview_orig_locstat_modified} 
illustrates that it is certainly not minimal among general Gaussian 
representations of $\gamma$. In each situation, the $\lambda$-modified 
processes \eqref{eq:fBMconvex} has an even smaller maximal variance on $K$ than
the $K$-stationary process derived from \eqref{eq:locstatK}. While this is not
a surprise for $\alpha=1$ or $\alpha=1.3$, where we know that \eqref{eq:fBMconvex}
is minimal, it seems quite remarkable how much the maximal variance of \eqref{eq:fBMconvex}
is reduced even in the case $\alpha=0.7$ compared to the $K$-stationary process from 
\eqref{eq:locstatK}. This is a major difference to the 1-dimensional case that we 
considered in Example~\ref{example:fBMgeneral}, where the $K$-stationary process 
provides the minimal representation.
\end{itemize}
\end{ex}

\begin{rem}
In the literature, there are different other terms related to the concept
of $K$-\emph{stationarity}. In \cite{analogies01}, for instance, a 
\emph{locally equivalent stationary covariance} is defined as covariance
function $C$ on $\RR^d \times \RR^d$ such that $C(\bm{x},\bm{y})$ only depends
on $\bm{x} - \bm{y}$ for \emph{all} $\bm{x}, \bm{y} \in \RR^d$ and  
\begin{equation} \label{eq:loc-equiv-stat-cov}
 C(\bm{x},\bm{y}) = C(\bm{0}) - \frac{1}{2}\gamma(\bm{x}-\bm{y}), \quad \bm{x}, \bm{y} \in K.
\end{equation}
In general, this assumption is stronger than the one in our definition, as it 
additionally requires that the locally defined covariance function $C$ on 
$K \times K$ can be extended to a stationary covariance function on 
$\RR^d \times \RR^d$. There are numerous examples of variograms $\gamma$ and 
corresponding functions of the form \eqref{eq:loc-equiv-stat-cov} that are 
positive definite on a small (potentially discrete and finite) domain $K$ for
sufficiently large $C(\bm{0})$, but cannot be extended to a positive definite
function on $\RR^d \times \RR^d$. In case of an isotropic variogram and a 
domain of the type $K = B_R(\bm{0})$ for some $R>0$, i.e.\ in the case 
considered in Proposition~\ref{prop:locstat}, however, such an extension is
always possible by Rudin's Theorem \citep{rudin70}. Thus, the terms 
\emph{locally equivalent stationary covariance} and \emph{covariance of a} 
$K$-\emph{stationary representation} are equivalent in case $K = B_R(\bm{0})$.
\end{rem}

\subsection{Discretization effects}
\label{sec:discrete}

In our examples above, we always considered processes on \emph{convex} 
$d$-dimensional domains $K$ with \emph{non-trivial interior} (such as 
$d$-dimensional hyperrectangles or $d$-dimensional balls of some radius). Indeed
the regions of interest are usually of such kind. 
On the other hand, the simulation domain $K$ in a computer experiment is 
necessarily a finite set of points $K=\{\bm{x}_i\}_{i=1,\dots,N}$ (such as a 
subgrid of a hyperrectangle) and the minimizing measure $\lambda_{\min}$ can
depend on this choice of discretization. In case that the variogram $\gamma$ is
convex, not much changes. Proposition~\ref{prop:convex} still applies and in 
particular, if the discritization consists of a subgrid of some rectangular 
domain, the minimizing measure $\lambda_{\min}$ will be equal to the uniform 
distribution on the vertices of the chosen subgrid. The modified process $W$ as
in \eqref{eq:fBMconvex} still yields the minimal representation of the variogram
$\gamma$.

It is again the case of a non-convex variogram $\gamma$, which is more 
intricate. First, we know that Proposition~\ref{prop:non-convex} also still 
applies and typically, we can reduce the maximal variance by a significant 
amount by the same strategy as before, i.e.\ by considering the modification $W$
in \eqref{eq:fBMconvex} instead of the original representation $W_0$. In the 
example of an $\alpha$-fractional Brownian sheet we found this strategy useful 
when $\alpha$ is close to 1. Secondly, since we discretized our space, it is 
usually possible to solve
\begin{align*}
\int_K \gamma(\bm{x}-\bm{y}) \lambda_0(\rd\bm{y}) = 1, \qquad \bm{x} \in K
\end{align*}
explicitly for $\lambda_0$. If $\lambda_0$ is a non-negative measure (and not a 
signed measure), normalizing $\lambda_0$ so that it becomes a probability 
measure will then provide the minimal ($K$-stationary) solution 
$\lambda_{\min}$, see \cite{matheron74}, page~8. In case of an 
$\alpha$-fractional Brownian sheet, we found this strategy useful for $d=1$ or
$\alpha$ close to zero and our numerical experiments led us to the following
conjecture.

\begin{conj}  \label{conj:criticalalpha}
Let $K=\{\bm{x}_i\}_{i=1,\dots,N}$ be a finite set of distinct points in 
$\RR^d$, which contains at least two elements. Set
$\bm{\Gamma}_{ij}=\lVert \bm{x}_i-\bm{x}_j \rVert^\alpha$, $i,j = 1,\dots,N$, 
$\alpha \in (0,1]$ and let $\bm{e} \in \RR^N$ be the vector, whose entries are
all equal to $1$.
\begin{enumerate}[label={(\alph*)},wide=0mm,itemsep=0mm]
\item For $d=1$ and $\alpha \in (0,1]$ all entries of $\bm{\Gamma}^{-1}\bm{e}$ 
     are non-negative (we write $\bm{\Gamma}^{-1}\bm{e}\geq 0$).
\item For $d\geq 2$ there exists $\alpha_{\text{\normalfont critical}} =\alpha_{\text{\normalfont critical}}(K)\in (0,1]$,
     such that $\bm{\Gamma}^{-1}\bm{e} \geq 0$ for $\alpha \in (0,\alpha_{\text{\normalfont critical}}]$.
\item We have $\alpha_{\text{\normalfont critical}}(K)=1$ if and only  if the
      points in $K$ are collinear.
\end{enumerate}
\end{conj}
 

For the remaining cases the minimal solution  $\lambda_{\min}=\sum_{i=1}^N \lambda_i\delta_{\bm{x}_i}$
can be obtained from the solution $\bm{\lambda} \in \RR^N$ of the optimization problem
\begin{align}\label{eq:discreteoptimalmeasure}
\min_{\bm{\lambda}} \quad \max_{i=1}^N \frac{1}{2} (\bm{\lambda} - \bm{e}^i)^\text{T} (-\bm{\Gamma}) (\bm{\lambda} - \bm{e}^i) \qquad \text{subject to} \qquad \bm{e}^\text{T} \bm{\lambda} = 1, \quad \bm{\lambda} \geq \bm{0},
\end{align}
where the vector $\bm{e}^i$ is the $i$-th unit vector in $\RR^N$, the matrix $\bm{\Gamma}$ has entries $\bm{\Gamma}_{ij}=\gamma(\bm{x}_i-\bm{x}_j)$ and $\bm{e}=\sum_{i=1}^N \bm{e}^i$.
This is a non-linear convex optimization problem that can be 
solved via standard optimization techniques.


\section{Numerical results}
\label{sec:study}

\begin{table}
\vspace*{1mm}
\centering
\caption{\small 
Simulation scenarios in dimension 1 on the interval $K=[-1,1]$ for Brown-Resnick
processes associated to the variogram $\gamma({h})=\lvert h/s \rvert^\alpha$, see
Section~\ref{sec:study}. The scale $s$ is obtained from \eqref{eq:scale}.}
\label{table:scenarios}
\vspace*{1mm}
\centering
\begin{tabular}{cccccc}
\toprule
\rowcolor{black!10}
& \bf Feature of   & \bf Minimal Gaussian
& \bf Scale 1 & \bf Scale 2 & \bf Scale 3
\\
\rowcolor{black!10}
& \bf variogram  & \bf representation & $\sigma^2_K=0.5$ & $\sigma^2_K=1$ & $\sigma^2_K=2$\\
\midrule
$\alpha=0.7$ & concave & $K$-stationary
& $s=0.818$ & $s=0.304$ & $s=0.113$ \\
$\alpha=1.0$ & linear  & $K$-stat.\ / $\lambda=\text{Unif}(\Ex(K))$ 
& $s=1.000$ & $s=0.500$ & $s=0.250$  \\
$\alpha=1.3$ & convex & $\lambda=\text{Unif}(\Ex(K))$ 
& $s=1.253$ & $s=0.735$ & $s=0.431$  \\
\bottomrule
\end{tabular}
\end{table}

Our numerical study focuses on Brown-Resnick processes on the hyperrectangle
$K=[-1,1]^d$ in dimensions $d \in \{1,2\}$ with the underlying variograms
belonging to the fractional Brownian sheet family
$\gamma(\bm{h})=\lVert \bm{h}/s \rVert^\alpha$. 
It compares the performance of the threshold stopping algorithms that are based
on
\begin{enumerate}[label={(\roman*)},itemsep=0mm]
\item the \emph{original} Gaussian representation $W_0$ of the variogram $\gamma$ with $W_0(\bm{0})=0$,
\item the minimal $K$-\emph{stationary} Gaussian representation from \eqref{eq:locstatK},
\item the $\lambda$-\emph{modified} Gaussian representation from \eqref{eq:fBMconvex}.
\end{enumerate}

In dimension 1 we fix the actual simulation domain as the grid $\{-1,-0.996,
\dots,1\} \subset [-1,1]=K$ that consists of 501 equally spaced points and vary
both the smoothness parameter $\alpha \in (0,2)$ and scale $s > 0$ as shown in
Table~\ref{table:scenarios}. In each column the scale is chosen such that the 
minimal $K$-stationary Gaussian representation of the variogram \eqref{eq:Cmin1}
has the same variance $\sigma^2_K$ across the domain $K=[-1,1]$, that is 
\begin{align}\label{eq:scale}
  s = \bigg(\frac{1}{{2}\sqrt{\pi} \sigma^2_{K}} {\Gamma\Big(\frac{2-\alpha}{2}\Big)\Gamma\Big(\frac{1+\alpha}{2}\Big)}\bigg)^{1/\alpha}, \quad   \sigma^2_K \in \{0.5,1,2\}.
\end{align}
In the case $\alpha=1$ one could 
equivalently fix the scale as $s=1$ and vary the domain $K=[-R,R]$ across $R=1$,
$R=2$ and $R=4$. In dimension 2 we fix the simulation domain as the grid 
$\{-1,-0.9,\dots,1\}^2 \subset [-1,1]^2=K$ that consists of $21 \times 21$ points
and consider as a variogram the classical two-dimensional Brownian sheet variogram  
$\gamma(\bm{h})=(2 \sqrt{2}/\pi) \lVert \bm{h}\rVert$ where $\alpha=1$. The 
scale $s=\pi/{(2 \sqrt{2})}$ is chosen in such a way that the variance 
$\sigma^2_K$ of the associated $K$-stationary representation from 
\eqref{eq:locstatK} equals 1.


All simulation experiments are repeated $50\,000$ times. Based on the 
$50\,000$ simulated realizations of $Z^{(T)}$ and $50\,000$ independent
realizations of the corresponding spectral process $V$, we estimate the 
benchmark simulation error $\mathcal{P}_{K,\tau}$ according to \eqref{eq:prob-error}.
 The estimated errors are reported in Table~\ref{table:simuresultsDim1} for 
dimension 1 and Table~\ref{table:simuresultsDim2} for dimension 2. The smallest
error term in each scenario is always marked bold. 
To ensure consistency across the different scenarios, we first fix 
$\tau_{\text{\upshape $\lambda$-mod}}$ such that the resulting benchmark error
term $\widehat{\mathcal{P}}_{K,\tau}$ assumes a value around 0.1, i.e.\ the 
third column in Tables~\ref{table:simuresultsDim1} and \ref{table:simuresultsDim2} 
is always fixed (up to small deviations). To ensure a fair comparison of the 
different algorithms, we then choose the other thresholds 
$\tau_{\text{\upshape orig}}$, $\tau_{\text{\upshape $K$-stat}}$ in such a way 
that the mean number $\sE T_{K,\tau}$ of Gaussian processes that need to be 
simulated to obtain one single approximation to the Brown-Resnick process stays 
fixed within each scenario (up to relative deviations smaller  than $1\%$).

\begin{table}[t]
\centering
\caption{\small Benchmark error terms ${\widehat{\mathcal{P}}}_{K,\tau}$ for threshold stopping algorithms based on Gaussian representations for the  simulation scenarios of Table~\ref{table:scenarios}, where $\alpha \in \{0.7,1.0,1.3\}$, see Section~\ref{sec:study}.
}
\label{table:simuresultsDim1}
\vspace*{1mm}
\begin{tabular}{ll|ccccc}
\toprule
\rowcolor{black!10}
&& \bf Original & $\bm{K}$\bf{-} & $\bm{\lambda=}\text{Unif}(\Ex(K))$ \\
\rowcolor{black!10}
\multicolumn{2}{c|}{\bf Scenario} 
& \bf definition  & \bf stationary & \bf modification \\
\midrule
\multirow{3}{*}{\rotatebox[origin=c]{90}{\bf Scale 1}} 
& $\alpha=0.7$ & 0.19 & \textbf{0.05} & 0.10 \\
& $\alpha=1.0$ & 0.17 & \textbf{0.10} & \textbf{0.10} \\
& $\alpha=1.3$ & 0.16 & 0.17 & \textbf{0.10} \\
\midrule
\multirow{3}{*}{\rotatebox[origin=c]{90}{\bf Scale 2}}
& $\alpha=0.7$ & 0.25 & \textbf{0.04} & 0.11 \\
& $\alpha=1.0$ & 0.22 & \textbf{0.11} & \textbf{0.11} \\
& $\alpha=1.3$ & 0.20 & 0.20 & \textbf{0.10} \\
\midrule
\multirow{3}{*}{\rotatebox[origin=c]{90}{\bf Scale 3}}
& $\alpha=0.7$ & 0.32 & \textbf{0.03} & 0.10 \\
& $\alpha=1.0$ & 0.24 & \textbf{0.11} & \textbf{0.09} \\
& $\alpha=1.3$ & 0.21 & 0.24 & \textbf{0.10} \\
\bottomrule
\end{tabular}
\vspace*{1mm}
\centering
\caption{\small Benchmark error terms ${\widehat{\mathcal{P}}}_{K,\tau}$ for threshold
 stopping algorithms based on Gaussian representations for the simulation of 
 Brown-Resnick processes on the square $K=[-1,1]^2$ associated to the variogram 
 $\gamma(\bm{h})=(2 \sqrt{2}/\pi) \lVert \bm{h}\rVert$.}
\label{table:simuresultsDim2}
\vspace*{1mm}
\begin{tabular}{ll|ccccc}
\toprule
\rowcolor{black!10}
&& \bf Original & $\bm{K}$\bf{-} & $\bm{\lambda=}\text{Unif}(\Ex(K))$ \\
\rowcolor{black!10}
\multicolumn{2}{c|}{\bf Scenario} 
& \bf definition & \bf stationary & \bf modification \\
\midrule
\multicolumn{2}{c|}{$\sigma^2_{K}=1$, $\alpha=1.0$} 
& 0.23 & 0.30 & \textbf{0.10} \\
\bottomrule
\end{tabular}
\end{table}


The results in Tables~\ref{table:simuresultsDim1} and \ref{table:simuresultsDim2} allow us to compare the 
performance of the different algorithms according to the error term 
${\widehat{\mathcal{P}}}_{K,\tau}$ {within} a given scenario $(\alpha,s)$.
As explained in Example~\ref{example:fBMgeneral}, in dimension 1 we know for 
each scenario $(\alpha,s)$ which of the Gaussian representations of the 
variogram $\gamma(h)=\lvert h/s\rvert^\alpha$ leads to the minimal maximal 
variance across the simulation domain $[-1,1]$. It only depends on the value of
$\alpha$ as listed in Table~\ref{table:scenarios}. Indeed the simulation 
results in Table~\ref{table:simuresultsDim1} show that the smallest error term
${\widehat{\mathcal{P}}}_{K,\tau}$ is always attained by the algorithm that corresponds to
the minimal representation. This confirms our previous theoretical 
considerations on the influence of the maximal variance of the spectral process
on the performance of the threshold stopping algorithms.
In dimension 2 we focus on the classical scenario of a two-dimensional Brownian
sheet. As we can see from Table~\ref{table:simuresultsDim2} and as anticipated, 
simulations via the threshold stopping algorithm based on the $\lambda$-modified
representation exhibit smaller errors than simulations based on the $K$-stationary representation and simulations based on the original definition. This observation  is well in line with our theoretical considerations, since the maximal variance of 
the  $\lambda$-modified representation is the smallest among those three 
(cf.\ Figure~\ref{fig:Dim2variances_a071013_overview_orig_locstat_modified}).
It is even minimal in the sense of Problem~\ref{prob:minvar}, 
cf.\ Proposition~\ref{prop:convex}. 
In case $\alpha=1.3$ in dimension~1 and $\alpha=1$ in dimension $2$, 
the threshold-stopping algorithm that is based on the $K$-stationary representation performs slightly worse than the original definition. Here, Figures~\ref{fig:Dim1variances} and \ref{fig:Dim2variances_a071013_overview_orig_locstat_modified} show (for both cases) that the variance of the original field is much smaller than the $K$-stationary variance  on a large proportion of the relevant domain $[-1,1]^d$, which can explain this phenomenon as a reasonable subasymptotic effect. Otherwise, the ranking of the threshold stopping algorithms according to the error terms even  corresponds precisely to the ranking of the maximal variance on $[-1,1]$ for each of the remaining scenarios in dimension 1.

In addition, we compare the performance of the above threshold stopping algorithms that are based on Gaussian processes to the performance of three other established algorithms for simulation of max-stable processes, namely
\begin{enumerate}[label={(\roman*)},itemsep=0mm,resume]
\item the \emph{random shift} approach (cf.~\cite{oks12}), a threshold stopping algorithm, where the original log-Gaussian spectral process $V^{\text{(orig)}}$ is additionally shifted uniformly across the (finite!)\ simulation domain $K=\{\bm x_1,\ldots,\bm x_N\}$, i.e.\ its spectral process is
\begin{align*}
V^{\text{(shift)}}(\bm{x})=V^{\text{(orig)}}(\bm{x}-\bm{S}), \quad \bm{x} \in K=\{\bm x_1,\ldots,\bm x_N\},
\end{align*}
where ${\bm S}$ is uniformly distributed on  $K=\{\bm x_1,\ldots,\bm x_N\}$ and independent of $V^{\text{(orig)}}$.
\item the \emph{Dieker-Mikosch} approach (cf.~\cite{dm15} and \cite{deo16}),
another  threshold stopping algorithm, which is based on the (sum-)normalized spectral process of the form
\begin{align} \label{eq:dm-spectral}
 V(\bm x) = N \frac{\exp(W^{\text{(orig)}}(\bm x - \bm S))}{\sum_{k=1}^N \exp(W^{\text{(orig)}}(\bm x - \bm x_k))}, \quad \bm x \in K=\{\bm x_1,\ldots,\bm x_N\},
\end{align}
where ${\bm S}$ is again uniformly distributed on $K=\{\bm x_1,\ldots,\bm x_N\}$ and independent of $W^{\text{(orig)}}$.
\item the \emph{extremal functions} approach (cf.~\cite{deo16}) with exact 
   simulation taking place on a subset of $N_{\text{\upshape extrfun}}$ 
   pre-specified equi-spaced locations in the simulation domain only.
\end{enumerate}

To ensure a fair comparison, the thresholds for the random shift approach and the Dieker-Mikosch algorithm are again chosen in such a way that the mean number $\sE T_{K,\tau}$ of Gaussian processes is fixed within each scenario (up to relative deviations smaller than $1\%$). Likewise, we fix the expected number of simulated Gaussian processes for the extremal functions approach, for which it corresponds to  the expected number $N_{\text{\upshape extrfun}}$ of locations for which exact simulation is ensured \citep[cf.][]{deo16}.
Since all scales in Table~\ref{table:simuresultsDim1} seem to reproduce a similar qualitative behaviour, we focus in this experiment on the scenarios in Table~\ref{table:scenarios} corresponding to Scale~1 ($\sigma_K^2=0.5$) and explore a slightly larger range of $\alpha \in \{0.1,0.4,\dots,1.9\}$. The results are reported in Table~\ref{table:simuresultsEFDM} and plotted in Figure~\ref{fig:simuresultsEFDM}.

\begin{table}
\caption{\small Benchmark error terms ${\widehat{\mathcal{P}}}_{K,\tau}$ for
 threshold stopping algorithms based on Gaussian representations and three further
 algorithms for the  simulation scenarios of Table~\ref{table:scenarios} with 
  Scale~1 ($\sigma_K^2=0.5$), see Figure~\ref{fig:simuresultsEFDM} for an illustration.
}
\label{table:simuresultsEFDM}
\vspace*{1mm}
\tabcolsep1.5mm
\begin{tabular}{ll|ccc|ccc}
\toprule
\rowcolor{black!10}
&& \bf Original & $\bm{K}$\bf{-} & $\bm{\lambda=}\text{Unif}(\Ex(K))$ 
& \bf Random & \bf Dieker- & \bf Extremal\\
\rowcolor{black!10}
\multicolumn{2}{c|}{\bf Scenario} 
& \bf definition  & \bf stationary & \bf modification 
& \bf shift & \bf Mikosch & \bf functions\\
\midrule
\multirow{7}{*}{\rotatebox[origin=c]{90}{{\bf Scale 1}}} 
& $\alpha=0.1$ & 0.244 & \textbf{0.007} & 0.098 
& 0.263 & \textbf{0.005} & 0.899 \\
& $\alpha=0.4$ & 0.204 & \textbf{0.018} & 0.095 
& 0.245 & \textbf{0.003} & 0.822 \\
& $\alpha=0.7$ & 0.186 & \textbf{0.051} & 0.103 
& 0.265 & \textbf{0.006} & 0.686 \\
& $\alpha=1.0$ & 0.168 & \textbf{0.102} & \textbf{0.101} 
& 0.236 & \textbf{0.017} & 0.521 \\
& $\alpha=1.3$ & 0.161 & 0.172 & \textbf{0.102} 
& 0.223 & \textbf{0.037} & 0.364 \\
& $\alpha=1.6$ & 0.130 & 0.241 & \textbf{0.102} 
& 0.171 & \textbf{0.063} & 0.231 \\
& $\alpha=1.9$ & 0.109 &  0.362 & \textbf{0.101}  
& 0.121 & \textbf{0.094} & 0.210 \\
\bottomrule
\end{tabular}
\end{table}

\begin{figure}[h]
\centering
\begin{tabular}{lc}
\rotatebox[origin=c]{90}{\begin{minipage}{7cm} \centering Average Error Probability ${\widehat{\mathcal{P}}}_{K,\tau}$ \end{minipage}} 
& \begin{minipage}{8.9cm}
\includegraphics[height=7cm]{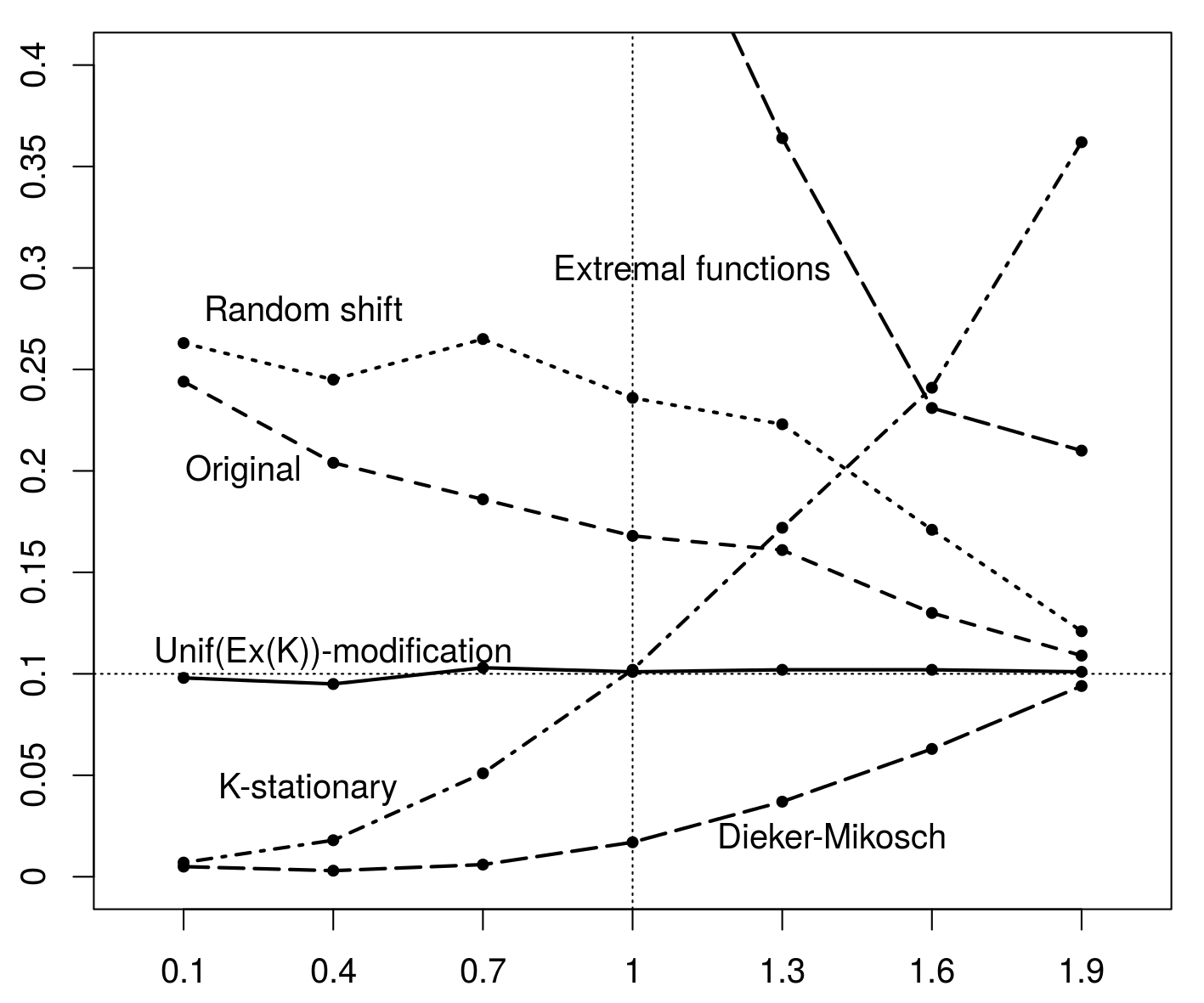}
\end{minipage}\\
& $\alpha$
\end{tabular}
\caption{\small Benchmark error terms ${\widehat{\mathcal{P}}}_{K,\tau}$ for threshold stopping algorithms based on Gaussian representations (``\textsf{Original}'', ``\textsf{Unif(Ex(K))-modification}'', ``\textsf{K-stationary}'') and three further  algorithms (``\textsf{Random shift}'', ``\textsf{Dieker-Mikosch}'', ``\textsf{Extremal functions}'') for the  simulation scenarios of Table~\ref{table:scenarios} with Scale~1 ($\sigma_K^2=0.5$), see Table~\ref{table:simuresultsEFDM} for the estimated values.}
\label{fig:simuresultsEFDM}
\end{figure}

A first observation (that we found surprising at first sight) is the poor performance of the extremal functions approach and the strong performance of the Dieker-Mikosch approach in this experiment. We did not expect this previously, since for \emph{exact} simulation the extremal functions approach always outperforms the Dieker-Mikosch algorithm in terms of expected number of simulated Gaussian processes, see \cite{deo16} and \cite{os18review}, i.e.\ one should prefer the extremal functions approach in this case. We conclude that allowing for a simulation error makes the Dieker-Mikosch approach competitive again with other methods. Intuitively, this can be explained by its spectral functions \eqref{eq:dm-spectral} being bounded and probabilistically homogeneous in space which makes its convergence rate 
be quite fast at the beginning. While the random shift approach also enforces homogeneity of the spectral functions, this is  however done in this case at the cost of increasing the maximal variance of the spectral functions, which explains why its performance is even worse than the performance of the original Gaussian representation.
Generally, the three threshold-stopping algorithms that are based on \emph{Gaussian} representation lie all inbetween the random shift approach (apart from too high values of $\alpha$, when the $K$-stationary approach is the worst) and the Dieker-Mikosch approach. Among them, one can also see the phase transition at $\alpha=1$ between the $K$-stationary approach performing better for $\alpha<1$ and the $\lambda$-modification taking over for $\alpha >1$ quite well. As $\alpha$ gets close to zero, the $K$-stationary approach and the Dieker-Mikosch approach show a similar error rate, whereas the $\lambda$-modification and the Dieker-Mikosch approach perform similarly as $\alpha$ gets close to $2$.

\section{Discussion}\label{sec:discussion}

Efficient simulation of Brown-Resnick processes is an important task that is 
often needed to describe the extremal behaviour of spatial random fields.
As exact simulation of such processes can be very time-consuming, in particular
when the simulation domain consists of a large number of points, it is often
necessary to resort to the simulation of approximations of these processes. 
This can be done by either cutting an exact algorithm short or by running an 
inexact stopping algorithm. A priori it is unclear which algorithm will lead
to the smallest error term for a fixed simulation time or vice versa. Our focus here lies on the classical threshold-stopping algorithm \citep{schlather02}, where we consider different (log-)Gaussian spectral representations.
Trying to mitigate the problem of the occurrence of exessively large variances of the spectral functions has lead us to the following optimization problem: 
\begin{center}
\emph{
\textbf{Minimization problem.} Among Gaussian 
  processes with prescribed variogram $\gamma$,\\ find a process $W$ such that 
  $\sup_{x \in K} \var(W(x))$ is minimal.
}
\end{center}
This is a difficult mathematical problem that has been of independent interest
for the case of $K$-stationary solutions in geostatistics (or locally equivalent
stationary solutions, respectively), see 
e.g.~\cite{analogies01,gneiting99,gneiting2000addendum,stein01,cd,matheron73}.
To the best of our knowledge, an explicit solution is known only in very few 
cases such as for the variogram of a fractional Brownian motion on an interval
(cf.~Example~\ref{example:fBMgeneral} stating \cite{matheron74}). Solutions 
also depend quite heavily on the geometry of the simulation domain $K$. 
Generally, \citeauthor{matheron74}'s (\citeyear{matheron74}) contribution does
not seem to have received very much attention in the literature so far. 

Here, we first make his description of a solution for convex variograms more explicit
for symmetric domains (cf.~Proposition \ref{prop:convex}) and second, we prove 
that the strategy that is employed for symmetric domains can even be applied to
most practically relevant non-convex variograms on hyperrectangular domains in
order to achieve a substantial reduction of the maximal variance 
$\sup_{x \in K} \var(W(x))$, albeit not a minimal solution (cf.~Proposition 
\ref{prop:non-convex}). We also consider discretization effects and conjecture for
the fractional Brownian sheet family with variogram $\lVert h \rVert^\alpha$ the 
existence of a critical value of $\alpha$, below of which the solution to the 
minimization problem is $K$-stationary (cf.~Conjecture~\ref{conj:criticalalpha}).

One of the nice features of the variance reduction is that it can be very easily implemented. Our simulation study confirms that  the proposed modification of threshold-stopping algorithms that are based on a Gaussian representation can lead to significant improvements of the probability that no approximation error occurs while the expected running time is fixed. 
We expect similar performance improvements also for other features, e.g.\ when other types of errors are considered, cf.\ \cite{os18review}. 
Further improvements are possible, when taking into account discretization effects, i.e.\ modifying the Gaussian representation of the variogram $\gamma$ according to \eqref{eq:minmeasure} using the discrete minimizing measure \eqref{eq:discreteoptimalmeasure}. In our numerical study, we always considered a relatively dense grid, such that discretization effects (not reported) did not play a significant role. However, we conclude from further experiments (not reported) that they can become relevant for coarser designs.

A comparison with three other established simulation algorithms demonstrates that our approach can compete with the (potentially exact)
extremal functions method \citep{deo16} and always outperforms the random shift modification \citep{oks12} due to its increased variance. 
We were unable to detect a scenario, in which our approach would outperform the Dieker-Mikosch algorithm \citep{dm15}. The latter seems
to converge relatively fast at the beginning stages of a simulation due  to the boundedness and probabilistic homogeneity of the spectral functions in this threshold-stopping approach.
However, when increasing the required accuracy, one should expect a phase transition, in which the Dieker-Mikosch approach performs better first, before the extremal functions approach  takes over again.
If we consider only the threshold-stopping approach based on \emph{Gaussian} spectral representations, the proposed variance reduction leads to the best performances as expected. Further comparisons of different simulation algorithms are beyond the scope of this paper, but will be addressed in \cite{os18review}.




\appendix

\section{Proofs}\label{sec:proofs}

\subsection{Proofs for Section~\ref{sec:specrep}}

\begin{prop} \label{prop:variance-tail}
 Let $\{W_1(\bm{x}), \bm{x} \in K\}$ and $\{W_2(\bm{x}), \bm{x} \in K\}$ be two
 centered Gaussian processes with a.s.\ bounded sample paths and bounded 
 variance functions $\sigma_i^2(\bm{x}) = \var(W_i(\bm{x}))$, $i=1,2$. 
 Further, let $\sigma_{1,\max}^2 < \sigma_{2,\max}^2$ where 
 $\sigma_{i,\max}^2 = \sup_{\bm{x} \in K} \sigma_i^2(\bm{x}) < \infty$ for $i=1,2$.
 Then, for any function $y \in C(K)$
 \begin{align*} 
\lim_{u \to \infty} \frac{\sP\big(\sup_{\bm{x} \in K} W_1(\bm{x}) - \sigma_1^2(\bm{x})/2 - y(\bm{x}) > u \big)}
                         {\sP\big(\sup_{\bm{x} \in K} W_2(\bm{x}) - \sigma_2^2(\bm{x})/2 - y(\bm{x})  > u \big)} = 0. 
\end{align*}
\end{prop}

\begin{proof}[Proof of \textbf{\upshape Proposition~\ref{prop:variance-tail}}]
By Proposition~1 in \citet{debicki-etal-10}, for $i=1,2$, we have
\begin{align*}
- \log \sP\bigg(\sup_{\bm{x} \in K} W_i(\bm{x}) - \sigma_i^2(\bm{x})/2 - y(\bm{x}) > u \bigg)
\sim  - \log \sP\bigg(\sup_{\bm{x} \in K} W_i(\bm{x}) > u \bigg)
\sim  \frac{u^2}{2 \sigma^2_{i,\max}} 
\end{align*}
 as $u \to \infty$. Consequently, for each $\varepsilon >0$, there is 
 some $u_i(\varepsilon) > 0$ such that for all $u > u_i(\varepsilon)$
 \begin{align*}
 \exp\bigg(-\frac{u^2(1+\varepsilon)}{2 \sigma_{i,\max}^2}\bigg) 
    \leq{} \sP\bigg(\sup_{\bm{x} \in K} W_i(\bm{x}) - \sigma_i^2(\bm{x})/2 - y(\bm{x}) > u \bigg) 
    \leq{} \exp\bigg(-\frac{u^2(1-\varepsilon)}{2 \sigma_{i,\max}^2}\bigg). 
 \end{align*}
 Now, let $\varepsilon > 0$ be sufficiently small such that 
 $(1-\varepsilon) \sigma^2_{2,\max} > (1+\varepsilon) \sigma^2_{1,\max}$. Then, for all
 $u > \max\{u_1(\varepsilon), u_2(\varepsilon)\}$, we have
 \begin{align*}
  \frac{\sP\big(\sup_{\bm{x} \in K} W_1(\bm{x}) - \sigma_1^2(\bm{x})/2 - y(\bm{x}) > u \big)}
       {\sP\big(\sup_{\bm{x} \in K} W_2(\bm{x}) - \sigma_2^2(\bm{x})/2 - y(\bm{x}) > u \big)}
  \leq{} & \exp\bigg( - \frac{u^2}{2} \cdot \bigg( \frac{1-\varepsilon}{\sigma^2_{1,\max}} - \frac{1+\varepsilon}{\sigma^2_{2,\max}} \bigg) \bigg),
 \end{align*}
 which implies the assertion.
\end{proof}

\subsection{Proofs for Section~\ref{sec:minloggauss}}

\begin{proof}[Proof of \textbf{\upshape Proposition~\ref{prop:convex}}]
First note that $\Ex(K)$ is a (locally compact) homogeneous space with respect
to the action of $\Sym(\Ex(K)) \subset O(\RR^d)$ and so a unique normalized left 
invariant Haar measure $\lambda$ exists on $\Ex(K)$ which we call uniform 
distribution, cf.~e.g.~\cite{nachbin76} or \cite{mardiakhatri77}.
Its support $\text{Supp}(\lambda)=\Ex(K)$ is necessarily a subset
of $\Ex(K)$, which establishes \eqref{eq:matheroncond1} for $\lambda$. 
Further, observe that the assignment
\begin{align*}
\bm{x} \mapsto \int_K \gamma(\bm{x}- \tilde{\bm{y}}) \, \lambda(\rd\tilde{\bm{y}}) 
= \int_{\Ex(K)} \gamma(\bm{x}- \tilde{\bm{y}}) \, \lambda(\rd\tilde{\bm{y}}) 
\end{align*}
is also a convex function (since $\gamma$ is convex). In particular, it attains
its maximal value $m$ on $K$ for some $v \in \Ex(K)$. Since 
$\Sym(\Ex(K))\subset O(\RR^d)$ acts transitively on $\Ex(K)$, any element in $\Ex(K)$ 
can be represented as $\bm{M}\bm{v}$ for some $\bm{M} \in O(\RR^d)$. Since 
$\gamma(\bm{h})$ depends only on $\lVert \bm{h} \rVert$, this gives
\begin{align*}
\int_{\Ex(K)} \gamma(\bm{M}\bm{v}- \tilde{\bm{y}}) \, \lambda(\rd\tilde{\bm{y}}) 
= \int_{\Ex(K)} \gamma(\bm{M}\bm{v}- \bm{M} \tilde{\bm{y}}) \, \lambda(\rd\tilde{\bm{y}}) 
= \int_{\Ex(K)} \gamma(\bm{v}- \tilde{\bm{y}}) \, \lambda(\rd\tilde{\bm{y}}).
\end{align*}
So, in fact, all elements of $\Ex(K)$ attain this maximal value $m$, which
implies
\begin{align*}
\int_K \int_K \gamma(\tilde{\bm{x}}- \tilde{\bm{y}}) \, \lambda(\rd\tilde{\bm{x}}) \, \lambda(\rd\tilde{\bm{y}}) 
= \int_{\Ex(K)} \int_{\Ex(K)} \gamma(\tilde{\bm{x}}- \tilde{\bm{y}}) \, \lambda(\rd\tilde{\bm{x}}) \, \lambda(\rd\tilde{\bm{y}}) 
= \int_{\Ex(K)}  m \lambda(\rd\tilde{\bm{y}}) = m.
\end{align*}
Finally, this gives
\begin{align*}
\int_K \gamma(\bm{x}- \tilde{\bm{y}}) \, \lambda(\rd\tilde{\bm{y}}) &\leq m \leq
\int_K \int_K \gamma(\tilde{\bm{x}}- \tilde{\bm{y}}) \, \lambda(\rd\tilde{\bm{x}}) \, \lambda(\rd\tilde{\bm{y}}) 
\end{align*}
for all $\bm{x} \in \Ex(K)$, as desired (cf.\ Condition \eqref{eq:matheroncond2}).
\end{proof}

\begin{proof}[Proof of \textbf{\upshape Lemma~\ref{lemma:Bernstein}}]
First note that (i) is equivalent to $\psi$ being negative definite in the 
sense of \eqref{eq:NDpsi}. The equivalence ``(i) $\Leftrightarrow$ (ii)'' is 
then immediate from Corollary~5.1.8 of \cite{bcr84} (page 150 therein) and the
fact that $(\RR,\RR_+,x^2)$ is a Schoenberg triple, cf.\ Example~5.1.3 of
\cite{bcr84} (page 146 therein). Second, the equivalence 
``(i) $\Leftrightarrow$ (iii)'' follows from Corollary~4.6.8 of \cite{bcr84}
(page 133 therein) and the 2-divisibility of the semigroup $(\RR_+,+)$.
\end{proof}

\begin{lem} \label{lemma:TWOalternating}
Let $\psi:\RR_+\rightarrow \RR$ be $n$-alternating up to $n\leq 2$ and 
$0 \leq x \leq R$ as well as $0 \leq x_i\leq R_i$ for $i=1,\dots,d$. Then the
following inequalities hold true for any $a\geq 0$.
\begin{enumerate}[label={(\alph*)}, wide=0mm]
\item $\psi(a + (R-x)^2) + \psi(a + (R+x)^2) \leq  \psi(a + R^2) + \psi(a + 3R^2)$
\item $\sum_{A \subset \{1,\dots,d\}} \psi(a+\sum_{i \in A} (x_i-R_i)^2 + \sum_{j \in A^c} (x_j+R_j)^2)$\\[2mm] 
$\phantom{(a)} \leq \sum_{A \subset \{1,\dots,d\}} \psi(a+3 \sum_{i \in A} R_i^2 + \sum_{j \in A^c} R_j^2)$ 
\end{enumerate}
\end{lem}

\begin{proof}
\begin{enumerate}[label={(\alph*)},wide=0mm]
\item By the alternation properties of $\psi$, we have that 
\begin{align*}
&\psi(a+(R-x)^2) + \psi(a + (R+x)^2) - \psi(a + R^2) - \psi (a + 3R^2)\\
\leq{}& \psi(a+(R-x)^2) + \psi(a + (R+x)^2) - \psi(a + R^2) - \psi (a + R^2 + 2x^2)\\
={}&\Delta_{2Rx-x^2} \Delta_{2Rx+x^2} \psi(a + (R-x)^2) \leq 0
\end{align*}
\item The assertion follows by induction on $d$. For $d=1$ it is evident from 
(a). For the step from $d$ to $d+1$ note that
\begin{align*}
&\sum_{A \subset \{1,\dots,d+1\}} \psi\bigg(a+\sum_{i \in A} (x_i-R_i)^2 + \sum_{j \in A^c} (x_j+R_j)^2\bigg)\\
={}& \sum_{B \subset \{1,\dots,d\}} \psi\bigg(a+\sum_{i \in B} (x_i-R_i)^2 + \sum_{j \in B^c} (x_j+R_j)^2+(x_{d+1}-R_{d+1})^2\bigg)\\
& + \sum_{B \subset \{1,\dots,d\}} \psi\bigg(a+\sum_{i \in B} (x_i-R_i)^2 + \sum_{j \in B^c} (x_j+R_j)^2+(x_{d+1}+R_{d+1})^2\bigg).
\end{align*}
According to (a) the latter is less than or equal to 
\begin{align*}
& \sum_{B \subset \{1,\dots,d\}} \psi\bigg(a+\sum_{i \in B} (x_i-R_i)^2 + \sum_{j \in B^c} (x_j+R_j)^2+3R_{d+1}^2\bigg)\\
& + \sum_{B \subset \{1,\dots,d\}} \psi\bigg(a+\sum_{i \in B} (x_i-R_i)^2 + \sum_{j \in B^c} (x_j+R_j)^2+R_{d+1}^2\bigg).
\end{align*}
Now, the induction hypothesis can be 
applied and gives that the latter is less than or equal to 
\begin{align*}
&\sum_{B \subset \{1,\dots,d\}} \psi\bigg(a+3\sum_{i \in B} R_i^2 + \sum_{j \in B^c} R_j^2+R_{d+1}^2\bigg)
+ \sum_{B \subset \{1,\dots,d\}} \psi\bigg(a+3\sum_{i \in B} R_i^2 + \sum_{j \in B^c} R_j^2+3R_{d+1}^2\bigg) \\
={}& \sum_{A \subset \{1,\dots,d+1\}} \psi\bigg(a+3\sum_{i \in A} R_i^2 + \sum_{j \in A^c} R_j^2\bigg).\qedhere
\end{align*}
\end{enumerate}
\end{proof}

\begin{lem} \label{lemma:THREEalternating}
Let $\psi:\RR_+\rightarrow \RR$ be $n$-alternating up to $n\leq 3$ and 
$\psi(0)=0$. Then we have that for $a,b\geq 0$
\begin{align*}
\psi(3 a + b)  - \frac{1}{2} \psi(4 a) \leq \psi(a + b). 
\end{align*}
\end{lem}

\begin{proof} 
The assertion follows from
\begin{align*}
&-\psi(a+b)+\psi(3a+b)-\frac{1}{2} \psi(4a) \\
={}&\big[\psi(a)-\psi(3a)-\psi(a+b)+\psi(3a+b)\big] \\
&+ \frac{1}{2}
\big[\psi(0)-\psi(2a)-\psi(a)-\psi(a)+\psi(3a)+\psi(3a)+\psi(2a)-\psi(4a)\big] \\
={}&\Delta_{2a}\Delta_{b}\psi(a) + \frac{1}{2} \Delta_{2a}\Delta_{a}\Delta_{a}\psi(0) \leq 0 \qedhere
\end{align*}
\end{proof}

\begin{proof}[Proof of \textbf{\upshape Proposition~\ref{prop:non-convex}}]
We need to show that 
\begin{align} \label{eq:non-convex-to-show}
\var\bigg(W_0(\bm{x}) - \frac{1}{2^d} \sum_{A \subset \{1,\dots,d\}} W_0(\bm{v}_A)\bigg) \leq \sup_{\bm{x} \in K} \var(W_0(\bm{x}))
\end{align}
for all $\bm{x} \in K$. Since
\begin{align*}
2 \,\cov(W_0(\bm{x})-W_0(\bm{y}),W_0(\bm{x'})-W_0(\bm{y'}))
= \gamma(\bm{x}-\bm{y'}) + \gamma(\bm{y}-\bm{x'}) - \gamma(\bm{x}-\bm{x'}) - \gamma(\bm{y}-\bm{y'}),
\end{align*}
the left-hand side of \eqref{eq:non-convex-to-show} equals
\begin{align*}
& \var\bigg(W_0(\bm{x}) - \frac{1}{2^d} \sum_{A \subset \{1,\dots,d\}} W_0(\bm{v}_A)\bigg)\\
={} & \frac{1}{2^{2d}} \sum_{A \subset \{1,\dots,d\}} \sum_{B \subset \{1,\dots,d\}} \cov\bigg( W_0(\bm{x}) - W_0(\bm{v}_A), W_0(\bm{x}) - W_0(\bm{v}_B) \bigg)\\
={} & \frac{1}{2^{2d+1}} \sum_{A \subset \{1,\dots,d\}} \sum_{B \subset \{1,\dots,d\}} \gamma( \bm{x} - \bm{v}_A ) + \gamma(\bm{x} - \bm{v}_B) + \gamma(\bm{v}_A - \bm{v}_B)\\
={} & \frac{1}{2^{d}} \sum_{A \subset \{1,\dots,d\}}  \gamma( \bm{x} - \bm{v}_A )  - \frac{1}{2} \gamma( \bm{v}_{\emptyset} - \bm{v}_A).
\end{align*}
On the other hand, since $W_0(\bm{0})=0$ and $\gamma(\bm{h})=\psi(\lVert\bm{h}\rVert^2)$
for a monotonously increasing function $\psi$, the right-hand side of 
\eqref{eq:non-convex-to-show} coincides with $\gamma(\bm{v}_A)$ for
any vertex $\bm{v}_A$ of $K$. Hence, we need to show that for all 
$\bm{x} \in K$
\begin{align*}
\frac{1}{2^{d}} \sum_{A \subset \{1,\dots,d\}}  \gamma( \bm{x} - \bm{v}_A )  - \frac{1}{2} \gamma( \bm{v}_{\emptyset} - \bm{v}_A) \leq \gamma(\bm{v}_{\emptyset} ), 
\end{align*}
which is equivalent to
\begin{align*}
\frac{1}{2^{d}} \sum_{A \subset \{1,\dots,d\}}  \psi\bigg( \sum_{i \in A} (x_i - R_i)^2 + \sum_{j \in A^c} (x_j + R_j)^2 \bigg)  - \frac{1}{2} \psi\bigg( 4 \sum_{i \in A} R^2_i\bigg) \leq \psi\big(R_1^2+\dots+R_d^2\big). 
\end{align*}
By Lemma~\ref{lemma:TWOalternating}, it suffices to show that 
\begin{align*}
\frac{1}{2^{d}} \sum_{A \subset \{1,\dots,d\}}  \psi\bigg( 3 \sum_{i \in A} R^2_i + \sum_{j \in A^c} R_j^2 \bigg)  - \frac{1}{2} \psi\bigg( 4 \sum_{i \in A} R^2_i\bigg) \leq \frac{1}{2^{d}} \sum_{A \subset \{1,\dots,d\}}
\psi\bigg(\sum_{i \in A} R^2_i + \sum_{j \in A^c} R_j^2\bigg). 
\end{align*}
The latter follows from Lemma~\ref{lemma:THREEalternating}.
\end{proof}

\textbf{Acknowledgements.} 
The authors would like to thank an anonymous referee for a lot of  well-conceived feedback on our work including the suggestion to include the Dieker-Mikosch algorithm in our simulation study. The authors would also like to thank Stilian Stoev and Holger Drees for their very helpful comments during EVA 2017 in Delft, which lead to rethinking the error assessment. The sole responsibility for all directions taken lies, of course, with the authors.

{

}

\end{document}